\pdfoutput=1
\documentclass{amsart}
\usepackage[utf8]{inputenc}

\usepackage[margin=1in]{geometry}
\usepackage[expansion=false]{microtype}

\usepackage{amsmath, amsfonts, amssymb, amsthm, amsopn, mathtools}
\usepackage{eucal}

\usepackage{tikz-cd}
\usetikzlibrary{patterns}
\usepackage{bm}
\usepackage{rotating}
\usepackage{xparse}
\usepackage{pict2e}

\usepackage[pdfusetitle,unicode,colorlinks]{hyperref}

\numberwithin{equation}{subsection}

\theoremstyle{plain}
\newtheorem{theorem}[equation]{Theorem}
\newtheorem{proposition}[equation]{Proposition}
\newtheorem{lemma}[equation]{Lemma}
\newtheorem{corollary}[equation]{Corollary}

\theoremstyle{definition}
\newtheorem{definition}[equation]{Definition}
\newtheorem{example}[equation]{Example}
\newtheorem{remark}[equation]{Remark}

\newtheorem{observation}[equation]{Observation}
\newtheorem{construction}[equation]{Construction}
\newtheorem{warning}[equation]{Warning}

\let\scr=\mathcal

\let\phi=\varphi

\let\into=\hookrightarrow
\let\onto=\twoheadrightarrow

\def\AA{\scr A}
\def\BB{\scr B}
\def\CC{\scr C}
\def\DD{\scr D}
\def\EE{\scr E}

\def\WW{\scr W}
\def\II{\scr I}

\def\OO{\scr O}

\def\QQ{\scr Q}

\def\SS{\scr S}

\def\UU{\scr U}

\def\WW{\scr W}
\def\XX{\scr X}
\def\YY{\scr Y}
\def\ZZ{\scr Z}

\DeclareMathOperator{\id}{id}
\DeclareMathOperator{\ev}{ev}

\DeclareMathOperator{\Cat}{Cat}
\DeclareMathOperator{\Pos}{Pos}

\DeclareMathOperator{\RFib}{RFib}
\DeclareMathOperator{\LFib}{LFib}

\DeclareMathOperator{\ICat}{\mathsf{Cat}}

\DeclareMathOperator{\IFilt}{\mathsf{Filt}}

\DeclareMathOperator{\Set}{Set}
\DeclareMathOperator{\Sub}{Sub}

\DeclareMathOperator{\Tw}{Tw}
\DeclareMathOperator{\Fun}{Fun}
\DeclareMathOperator{\Map}{map}

\DeclareMathOperator{\const}{const}
\DeclareMathOperator*{\diag}{diag}

\DeclareMathOperator{\Sp}{Sp}

\DeclareMathOperator{\Fin}{Fin}
\DeclareMathOperator{\IFin}{\mathsf{Fin}}
\DeclareMathOperator{\fin}{fin}

\DeclareMathOperator{\Cone}{Cone}

\DeclareMathOperator{\et}{\Acute{e}t}

\DeclareMathOperator{\Disc}{Disc}

\DeclareMathOperator{\hyp}{hyp}
\DeclareMathOperator{\sh}{sh}
\DeclareMathOperator{\Id}{Id}

\newcommand{\col}[1]{\operatorname{colim}^{\operatorname{#1}}}
\newcommand{\LPr}{\operatorname{Pr}^{\operatorname{L}}}
\newcommand{\RPr}{\operatorname{Pr}^{\operatorname{R}}}

\newcommand{\RPrS}{\RPr_\infty}

\newcommand{\LTop}{\operatorname{Top}^{\operatorname{L}}}
\newcommand{\RTop}{\operatorname{Top}^{\operatorname{R}}}

\newcommand{\RTopS}{\RTop_\infty}

\newcommand{\map}[1]{\Map_{#1}}

\newcommand{\Over}[2]{#1_{\hspace{-1pt}/#2}}
\newcommand{\Under}[2]{#1_{\hspace{-1pt}#2/}}

\newcommand{\I}[1]{\mathsf{#1}}

\newcommand{\Comma}[3]{{#1}\downarrow_{#2}{#3}}

\newcommand{\compact}{\operatorname{cpt}}
\newcommand{\CatS}{\Cat_{\infty}}
\newcommand{\CatSS}{\widehat{\Cat}_\infty}

\newcommand{\cc}{\text{\normalfont{cc}}}

\newcommand{\lex}{\text{\normalfont{lex}}}

\newcommand{\ihom}{\underline{\operatorname{hom}}}
\newcommand{\iFun}[1][\BB]{\underline{\mathsf{Fun}}_{#1}}
\newcommand{\IPSh}[1][\BB]{\underline{\mathsf{PSh}}_{#1}}
\newcommand{\PSh}[1][\SS]{\operatorname{PSh}_{#1}}

\newcommand{\IShv}[1][\BB]{\underline{\mathsf{Sh}}_{#1}}
\newcommand{\Shv}[1][]{\operatorname{Sh}_{#1}}

\NewDocumentCommand{\Gen}{m o}{%
	\IfNoValueTF{#2}{%
		\langle #1\rangle%
	}{%
		\langle #1\rangle_{#2}%
	}%
}

\NewDocumentCommand{\Univ}{o}{%
	\IfNoValueTF{#1}{%
		\I{\Omega}%
	}{%
		\I{\Omega}_{#1}%
	}%
}
\NewDocumentCommand{\UnivHat}{o}{%
	\IfNoValueTF{#1}{%
		\widehat{\I{\Omega}}%
	}{%
		\widehat{\I{\Omega}}_{#1}%
	}%
}

\newcommand{\op}{\mathrm{op}}
\newcommand{\core}{\simeq}

\let\lim=\relax
\DeclareMathOperator*{\lim}{lim}
\DeclareMathOperator*{\colim}{colim}

\makeatletter
\g@addto@macro\bfseries{\boldmath}
\makeatother

\newtheoremstyle{introthms}
{}{}{\itshape}{}{\bfseries }{}{ }
{\thmname{#1} \thmnumber{#2}. \thmnote{\bfseries{(#3)}}}

\theoremstyle{introthms}
\newtheorem{introthm}{Theorem}

\title{Proper morphisms of $\infty$-topoi}
\author{Louis Martini}
\author{Sebastian Wolf}
\date{\today}

\begin{document}

\begin{abstract}
	We characterise proper morphisms of $\infty$-topoi in terms of a relativised notion of compactness: we show that a geometric morphism of $\infty$-topoi is proper if and only if it commutes with colimits indexed by filtered internal $\infty$-categories in the target. In particular, our result implies that for any $\infty$-topos, the global sections functor is proper if and only if it preserves filtered colimits. As an application, we show that every proper and separated map of topological spaces gives rise to a proper morphism between the associated sheaf $\infty$-topoi, generalising a result of Lurie.
\end{abstract}

\maketitle
\setcounter{tocdepth}{2}
\tableofcontents
\addtocontents{toc}{\protect\setcounter{tocdepth}{1}}

\section{Introduction}

\subsection*{Compactness in topos theory}
A fundamental principle of topos theory is that topoi can be conceptualised as generalised topological spaces.
As a result, every topological concept ought to have a counterpart in topos theory. 
One of the core notions in geometry and topology is the concept of \emph{compactness}. Thus, the natural question arises what a toposic analogue of this notion might be.

A sanity check for a toposic definition of compactness is that it must faithfully extend the definition of compactness for topological spaces, in the sense that a space $X$ is compact if and only if its topos of sheaves $\Shv[](X)$ is compact.
Now it turns out that the notion of compactness of $ X $ can be formulated on the level of the underlying locale $ \OO(X) $ of open subsets of $ X $: the space $ X $ is compact if and only if the global sections functor $ \OO(X) \to \OO(\ast) $ preserves filtered colimits. This leads to a natural candidate for the notion of compactness of topoi: a topos $\XX$ is said to be compact if the global sections functor $\XX\to \Set$ preserves filtered colimits of subterminal objects.

However, this definition also suggests an alternative, stronger compactness condition: 
one could also require that the global sections functor $\XX \to \Set $ preserves \emph{all} filtered colimits.
It is no longer true that the sheaf topos associated to every compact topological space has this property (see Example~\ref{ex:SeparatedNecessary}), but it does hold whenever the space is in addition Hausdorff \cite[Example 1.1 (3)]{Moerdijk2000}.
These two notions of toposic compactness have been studied extensively in \cite{Moerdijk2000} by Moerdijk and Vermeulen under the names \emph{compact} and \emph{strongly compact}, respectively.

From a modern point of view, it seems unreasonable to only consider ($1$-)topoi as generalised spaces, since many invariants of topological spaces arise more naturally on the level of their associated sheaf $\infty$-topoi. Thus, it seems more appropriate to consider $\infty$-topoi as the correct generalisation of topological spaces. 
For $\infty$-topoi, there is now a whole infinite hierarchy of compactness conditions: 
we call an $\infty$-topos $ \XX $ $ n $\emph{-compact} if the global sections functor $ \XX \to \SS$ commutes with filtered colimits of $ n $-truncated objects. In the case $n=\infty$, we simply say that $\XX$ is \emph{compact}. Thus, for $n=-1$ and $ n= 0$ we recover the notions of compactness and strong compactness from above, but for $n\geq 1$ we obtain something new. 

\subsection*{Properness in topos theory}
In topology, there is also a \emph{relative} notion of compactness: that of \emph{properness} of a map of topological spaces. Again, a natural question is what the ($\infty$-)toposic analogue of this concept might be. A possible approach to such a toposic definition of properness is motivated by the observation that proper maps of topological spaces satisfy the \emph{proper base change theorem}. The idea is now to take the proper base change theorem as a \emph{characterisation}, which leads to the following definition:
A geometric morphism $ p_* \colon \XX \to \BB $ of $ \infty $-topoi is called \emph{$n$-proper} if for any commutative square
\[\begin{tikzcd}
	{\WW'} & \WW & \XX \\
	{\ZZ'} & \ZZ & \BB
	\arrow["{p_*}", from=1-3, to=2-3]
	\arrow["{f_*}"', from=2-2, to=2-3]
	\arrow["{q_*}"', from=1-2, to=2-2]
	\arrow["{g_*}", from=1-2, to=1-3]
	\arrow["{f'_*}"', from=2-1, to=2-2]
	\arrow["{g'_*}", from=1-1, to=1-2]
	\arrow["{r_*}"', from=1-1, to=2-1]
\end{tikzcd}\]
of $ \infty $-topoi in which both squares are cartesian, the canonical natural transformation $ f'^* q_* \to r_* g'^* $ is invertible when restricted to $n$-truncated objects. For $n=\infty$, we simply say that $p_\ast$ is \emph{proper}, which is the case considered by Lurie in~\cite[\S~7.3.1]{htt}. For $n=-1$ and $n=0$, this definition recovers the notions of \emph{properness} and \emph{tidiness} of a morphism of $1$-topoi as studied by Moerdijk and Vermeulen~\cite{Moerdijk2000}. The fact that these definitions are reasonable is evidenced by the observation that if $p\colon Y\to X$ is a proper map of topological spaces, then the induced geometric morphism $p_\ast\colon\Shv(Y)\to\Shv(X)$ is $0$-proper~\cite[Example~I.1.9]{Moerdijk2000}, and it is $1$-proper if $p$ is in addition \emph{separated}~\cite[Example~II.1.4]{Moerdijk2000}.
if $Y$ is furthermore completely regular (i.e. a subspace of a compact Hausdorff space), then $p_\ast$ is even proper~\cite[Theorem~7.3.1.16]{lurie2009b}.

\subsection*{Main results}
The notion of properness of a map of $\infty$-topoi is supposed to capture a relative version of compactness. In particular, if the codomain is the final $\infty$-topos $\SS$ of $\infty$-groupoids, one expects to recover the notion of a compact $\infty$-topos. In other words, the global sections functor $\XX\to\SS$ should be proper if and only if $\XX$ is compact. However, this fact is not at all evident from the definitions, and it is one of the main goals in this paper to show that these two notions do in fact agree:
\begin{introthm}
	\label{thm:properBaseChangeOverSpaces}
	Let $ \XX $ be an $ \infty $-topos.
	Then the global sections functor $\Gamma_{\XX}\colon  \XX \to \SS $ is proper if and only if $\XX$ is a compact $\infty$-topos.
\end{introthm}

Theorem~\ref{thm:properBaseChangeOverSpaces} is a special case of a more general result which characterises properness of an arbitrary geometric morphism in terms of a relativised notion of compactness.
To make this precise, we need to invoke the theory of $ \infty $-categories internal to an $ \infty $-topos, developed in \cite{Yoneda, Colimits, Cocartesian, PresTop}.
An $ \infty $-category internal to an $ \infty $-topos $ \BB $ can simply be defined as a sheaf of $ \infty $-categories on $ \BB $, i.e. a limit-preserving functor $ \BB^\op \to \Cat_\infty $.
We refer to $ \infty $-categories internal to $ \BB $ as $ \BB $\emph{-categories}. Now to each geometric morphism $f_\ast\colon\XX\to\BB$, one can associate such a $\BB$-category via the assignment $A\mapsto \Over{\XX}{f^\ast(A)}$.
For reasons we will explain later, we denote this $ \BB $-category by $ f_* \Univ[\XX] $. This construction defines a functor $ {(\RTopS)}_{/\BB} \to \Cat(\BB) $, where $\RTopS$ denotes the $\infty$-category of $\infty$-categories and $\Cat(\BB)$ is the $\infty$-category of $\BB$-categories. As a consequence, if we denote the image of $ \id_\BB $ along this functor by $ \Univ_\BB $, the map $f_\ast$ itself determines a morphism of $ \BB $-categories $ f_* \Univ[\XX] \to \Univ_\BB $, which can be regarded as the \emph{internal} global sections functor of the $\BB$-category $f_\ast\Univ[\XX]$.

In \cite[\S 2.1]{PresTop} we introduced the notion of a \emph{filtered} $ \BB $-category (see also Example~\ref{ex:twoExOfInternalClasses}).
Therefore we can now define a geometric morphism $ p_* \colon \XX \to \BB $ to be \emph{compact} if the internal global sections functor $p_\ast\Univ[\XX]\to\Univ[\BB]$ commutes with such internally filtered colimits (see Example~\ref{ex:twoExOfInternalClasses} (3)).
Our main result is now the following generalisation of Theorem~\ref{thm:properBaseChangeOverSpaces}\footnote{That such a theorem should hold is already remarked in \cite[Remark 7.3.1.5]{htt}}:

\begin{introthm}
	\label{thm:MainTheorem}
	A geometric morphism $p_* \colon \XX \to \BB $ is proper if and only if it is compact.
\end{introthm}

The 1-categorical analogue of this result is due to T.\ Lindgren \cite{lindgren1984} and appears in \cite[Theorem 4.8]{Moerdijk2000}.
However we want to emphasise that proper morphisms of $ \infty $-topoi are not just a generalisations of tidy morphisms of $ 1$-topoi.
Instead, the condition of properness is genuinely stronger.
For example, for a compact $ 1 $-topos $ \XX $, the associated $1$-localic $ \infty $-topos $ \Shv(\XX) $ is often not compact (see Remark~\ref{rem:Spec(R)Counterexample}).

Theorem~\ref{thm:MainTheorem} makes it possible to enlarge the class of examples from topology which give rise to proper maps of $\infty$-topoi. As we mentioned above, Lurie showed that every proper map $p\colon Y\to X$ of topological spaces in which $Y$ is completely regular gives rise to a proper morphism between the associated sheaf $\infty$-topoi. However, since (toposic) properness is an entirely relative notion, it is somewhat surprising that there are constraints on the space $Y$ and not just on the map $p$. Instead, in light of Lurie's result that the $\infty$-topos of sheaves on a compact Hausdorff space is compact, one would expect that every proper and \emph{separated} map $p\colon Y\to X$ of topological spaces gives rise to a proper morphism of $\infty$-topoi\footnote{Further evidence that such a statement is expected to be true is that proper and separated maps of topological spaces are known to satisfy the proper base change theorem with abelian coefficients~\cite{Schnuerer2016}.}. We will leverage Theorem~\ref{thm:MainTheorem} in order to show that this is indeed the case:
\begin{introthm}
	\label{thm:ProperMapsTopology}
	Every proper and separated map $p\colon Y\to X$ of topological spaces gives rise to a proper morphism of $\infty$-topoi $p_\ast\colon \Shv(Y)\to\Shv(X)$.
\end{introthm}
Note that if $p\colon Y\to X$ is proper, then $Y$ being completely regular forces $p$ to be separated. Therefore, Theorem~\ref{thm:ProperMapsTopology} constitutes a generalisation of Lurie's result.

In contrast to the topological case, a proper morphism $p\colon X\to S$ of schemes usually does not give rise to a proper map $p_\ast\colon X_{\et}^{\hyp}\to S_{\et}^{\hyp}$ between the associated $\infty$-topoi of \'etale hypersheaves. Hence, the relevance of Theorem~\ref{thm:MainTheorem} in algebraic geometry is somewhat limited. However, the geometric morphism $p_\ast$ does satisfy proper base change when using coefficients in the derived $\infty$-category $\mathbf{D}(R)$ of any torsion ring $R$. Thus, one might speculate that $p_\ast$ satisfies an $R$-linear version of properness. Such a result would easily follow from a variant of Theorem~\ref{thm:MainTheorem} in which the notions of properness and compactness are replaced by certain $\mathbf{D}(R)$-linear versions thereof, which are weaker than plain properness and compactness. Hence, it could be of interest for applications in algebraic geometry to generalise Theorem~\ref{thm:MainTheorem} by allowing coefficients in any presentable $\infty$-category $\EE$. This leads us to our last major result in this paper: there are evident generalisations of properness and compactness of a geometric morphism to $\EE$-properness and $\EE$-compactness, respectively, and these two notions turn out to be equivalent under mild conditions on $\EE$:
\begin{introthm}
	\label{thm:properMorphismsCoefficients}
	For every compactly generated $\infty$-category $\EE$, a geometric morphism $p_\ast\colon \XX\to\BB$ is $\EE$-proper if and only if it is $\EE$-compact.
\end{introthm}

\subsection*{Outline of contents}
We begin in Chapter~\ref{sec:Recollection} by recalling the basics of $ \BB $-category theory.
The goal is to gather just enough material so that also the reader who isn't familiar with our previous work~\cite{Yoneda, Colimits, Cocartesian, PresTop} can follow our arguments.

Most of Chapter~\ref{sec:ProperMorphisms} is dedicated to the study of proper and compact morphisms of $\infty$-topoi and the proof of Theorem~\ref{thm:MainTheorem}.
Let us briefly outline the strategy of the proof.
It is mostly an adaption of the strategy for the $ 1 $-toposic result, given in \cite[p. 673-676]{johnstone2002}.
However, there are some complications coming from the fact that as opposed to the $1$-toposic case, we cannot rely on the existence of sites.
The non-trivial part of the theorem is to show that every pullback square
\[\begin{tikzcd}
	{\WW} & \XX \\
	\ZZ & \BB
	\arrow["{p_*}", from=1-2, to=2-2]
	\arrow["q_\ast", from=1-1, to=2-1]
	\arrow["g_\ast", from=1-1, to=1-2]
	\arrow["{f_*}"', from=2-1, to=2-2]
\end{tikzcd}\]
in $ \RTopS$ is left adjointable whenever $p_\ast$ is compact. We can reduce to the case where $f_\ast$ is fully faithful as follows: as explained above, the functor $f_\ast$ determines a map $f_\ast\Univ[\ZZ]\to \Univ[\BB]$ of $\BB$-categories. By using the theory of $\BB$-topoi that we developed in~\cite{PresTop}, we may find a $\BB$-category $\I{C}$ such that this map factors into a composition $f_\ast\Univ[\ZZ]\into\iFun(\I{C},\Univ[\BB])\to\Univ[\BB]$ (where $\iFun(-,-)$ is the internal hom of $\BB$-categories) in which the first map is fully faithful and the second map is given by taking internal limits. By evaluating these maps of $\BB$-categories at the final object $1\in\BB$, one thus obtains a factorisation of $f_\ast$ into the composition $ \ZZ \into \Fun_\BB(\I{C},\Univ[\XX]) \to \BB $, splitting up our pullback square into two.
For the right square, base change will be essentially obvious and after replacing $ \BB $ with $ \Fun_\BB(\I{C},\Univ[\BB]) $ (which is again an $\infty$-topos), we have thus reduced to the case where $f_\ast$ is the inclusion of a subtopos.
In this particular case, it turns out that one can write both $ f^* $ and $g^\ast$ as internal colimits indexed by certain filtered $ \BB $-categories which are compatible in a suitable way. Since $ p_* $ is compact, this functor commutes with these colimits, which allows us to deduce the base change isomorphism in this case.
An interesting aspect of our strategy is that, even in the special case of Theorem~\ref{thm:properBaseChangeOverSpaces}, our proof heavily relies on using methods from internal higher category theory.

In \S~\ref{sec:Topology}, we discuss our main application of Theorem~\ref{thm:MainTheorem}, that every proper and separated map of topological spaces gives rise to a proper morphism of $\infty$-topoi, i.e.\ Theorem~\ref{thm:ProperMapsTopology}. 
As mentioned above, Lurie already proves this result in the special case where the domain of the map is completely regular. Our proof strategy is quite different and arguably more direct.

Finally, in~\S~\ref{sec:coefficients} we discuss how our arguments for the proof of Theorem~\ref{thm:MainTheorem} can be generalised to show the analogous version with coefficients in a compactly generated $\infty$-category, i.e.\ Theorem~\ref{thm:properMorphismsCoefficients}.

\subsection*{Acknowledgements}
We would like to thank our respective advisors Rune Haugseng and Denis-Charles Cisinksi for their advice and support while writing this article.
We are also grateful to Mathieu Anel and Jonathan Weinberger for interesting discussions about the content of this paper.
In particular, they pointed out a critical mistake in an earlier draft of this article and suggested a strategy for correcting it. We could not have completed the paper without their invaluable help.
Furthermore, we would like to thank Ko Aoki and Marco Volpe for providing useful insights about applications of the main theorem in topology.
The first-named author was partially supported by the project Pure Mathematics in Norway, funded by Trond Mohn Foundation and Tromsø Research Foundation.
The second-named author gratefully acknowledges support from the SFB 1085 Higher Invariants in Regensburg, funded by the DFG.

\addtocontents{toc}{\protect\setcounter{tocdepth}{2}}

\section{Recollection of $\BB$-category theory}
\label{sec:Recollection}
In this chapter we give a brief account on the language of \emph{$\BB$-categories}, i.e.\ the theory of $\infty$-categories internal to an $\infty$-topos $\BB$. We will limit our exposition to only those aspects of the framework that we will make use of in the main body of the paper. For a more comprehensive discussion, we refer the reader to our paper series~\cite{Yoneda, Colimits, Cocartesian, PresTop}.

\subsection{Main definitions and examples}
We begin our recollection by providing the definition of $\BB$-categories as well as the most important procedures to construct examples of these objects.
\begin{definition}[{\cite[Definition~3.2.4 and~Proposition~3.5.1]{Yoneda}}]
	\label{def:BCats}
	A \emph{$ \BB $-category} $ \I{C} $ is a sheaf of $ \infty $-categories on $ \BB $, i.e. a (small) limit-preserving functor $\I{C}\colon \BB^\op\to\CatS$. Similiarly, a sheaf of large $ \infty $-categories $\I{D} \colon \BB^\op \to \CatSS$ is called a \emph{large} $ \BB $-category. 	If no confusion can arise we will often omit specifying the relative size. 
	A functor $f\colon \I{C}\to\I{D}$ between two $ \BB $-categories is simply a natural transformation of sheaves.
	We will denote the $ \infty $-category of $ \BB $-categories by $ \Cat(\BB) = \Fun^{\lim}(\BB^\op,\CatS) $.
\end{definition}

\begin{remark}
	By the universal property of the $ \infty $-category of spaces $ \mathcal{S} $, we have an equivalence 
	\[
		\Fun^{\lim}(\mathcal{S}^{\op}, \CatS)	\simeq \Cat_\infty,
	\]
	so that an $ \SS $-category is just an $ \infty $-category.
\end{remark}

\begin{example}\label{ex:MainList}
	There are a number of general procedures to construct examples of $\BB$-categories:
	 \begin{enumerate}
	 	\item Every object $ A \in \BB $ gives rise to a $ \BB $-category by means of the Yoneda embedding
	 	\begin{equation*}
		h_{\BB} \colon \BB \simeq \Fun^{\lim}(\BB^\op,\SS) \into \Cat(\BB).
	 	\end{equation*}
	 	We refer to the $\BB$-categories which are contained in the essential image of $h_{\BB}$ as \emph{$\BB$-groupoids}. To every $\BB$-category $\I{C}$ one can associated its \emph{core} $\BB$-groupoid $\I{C}^\core$ by setting $\I{C}^\core(-)= \I{C}(-)^\core$.
		\item Any geometric morphism $ (f^\ast\dashv f_\ast)\colon \colon \AA \leftrightarrows \BB $ of $ \infty $-topoi gives rise to an adjunction
			\[
				\begin{tikzcd}[column sep={7em,between origins}]
					\Cat(\AA) \arrow[r, bend left=30, shift left=4pt, "f^\ast"{name=U}, start anchor=east, end anchor=west] \arrow[from=r, bend left=30, shift left=4pt, "f_\ast"{name=D}, start anchor=west, end anchor=east]\arrow[from=U, to=D, phantom, "\bot"]& \Cat(\BB)
				\end{tikzcd}
			\]
		in which the right adjoint is given by precomposition with $f^\ast\colon \BB\to \AA$ and in which the left adjoint is given by left Kan extension along $f^\ast\colon \BB\to \AA$~\cite[\S~3.3]{Yoneda}.
		
		Specialising the above to the terminal geometric morphism $(\const_{\BB}\dashv\Gamma_{\BB}) \colon \SS\leftrightarrows\BB $, we obtain an adjunction $(\const_{\BB}\dashv\Gamma_{\BB})\colon  \CatS \leftrightarrows \Cat(\BB) $.
		For an arbitrary $ \infty $-category $ \CC $, we refer to the associated $\BB$-category $ \const_{\BB} (\CC) $ as the \emph{constant $ \BB $-category with value $ \CC $}. We will often implicitly identify an $\infty$-category with the associated constant $\BB$-category if there is no chance of confusion.
		
		Specialising the above to the \emph{\'etale geometric morphism} $(\pi_A^\ast\dashv(\pi_A)_\ast)\colon \BB\leftrightarrows\Over{\BB}{A}$ associtated to an object $A\in\BB$, we obtain an induced adjunction $(\pi_A^\ast\dashv(\pi_A)_\ast)\colon \Cat(\BB)\leftrightarrows\Cat(\Over{\BB}{A})$. For any $\BB$-category $\I{C}$, we refer to $\pi_A^\ast\I{C}$ as the \emph{\'etale base change} of $\I{C}$. Note that in this situation, the functor $\pi_A^\ast\colon\Cat(\BB)\to\Cat(\Over{\BB}{A})$ is simply given by precomposition with the forgetful functor $(\pi_A)_!\colon\Over{\BB}{A}\to\BB$.
		\item The $ \infty $-category $ \Cat(\BB) $ is presentable and cartesian closed~\cite[Propositions~3.2.9 and~3.2.11]{Yoneda}.
		It follows that for any two $ \BB $-categories $ \I{C} $ and $ \I{D} $ we can form a new $ \BB $-category $ \iFun(\I{C},\I{D}) $, which we call the \emph{$ \BB $-category of functors from $ \I{C} $ to $ \I{D} $}.
		Furthermore we write $\Fun_\BB(\CC,\DD) = \Gamma_{\BB} \iFun(\I{C},\I{D})$ for the underlying \emph{$ \infty $-category} of functors from $ \I{C} $ to $ \I{D} $.
		\item Unstraightening the codomain fibration $ d_0 \colon \Fun(\Delta^1,\BB) \to \BB $ yields a functor 
		\[
			\BB^\op \to \CatSS\; \; \; A \mapsto \BB_{/A},
		\]
		which preserves small limits since $ \BB $ is an $ \infty $-topos (\cite[Theorem 6.1.3.9]{htt}) and therefore defines a large $\BB$-category.
		We denote this $ \BB $-category by $ \Univ[\BB] $ (or simply $\Univ$ if the ambient $\infty$-topos is clear from the context) and refer to it as the \emph{universe} of $ \BB $.
		We will see below that $ \Univ[\BB] $ is the $\BB$-categorical analogue of the $ \infty $-category of spaces $\SS$.
		
		More generally, every \emph{local class} $S$ of morphisms in $\BB$ (in the sense of~\cite[Definition~6.1.3.8]{htt}) determines a full subcategory $\Univ[S]\into \Univ[\BB]$ (i.e.\ a full subsheaf) that is given for $A\in\BB$ by the full subcategory of $\Over{\BB}{A}$ spanned by the maps $P\to A$ which are contained in $S$. Conversely, every full subsheaf of $\Univ[\BB]$ determines a local class in $\BB$, and the two operations are mutually inverse~\cite[\S~3.9]{Yoneda}.
		\item 
		For any presentable $ \infty $-category $ \DD $, the functor
		\[
			\DD \otimes \Univ[\BB] \colon \BB^\op \to \CatSS; \; \; A \mapsto \BB_{/A} \otimes \DD\simeq\Fun^{\lim}((\Over{\BB}{A})^\op,\DD)
		\]
		(where $-\otimes-$ is the tensor product in $\RPrS$) defines a large $ \BB $-category \cite[Construction A.0.1]{Colimits}. Specialising this construction to $\DD=\CatS$, one obtains the $\BB$-category $\ICat_{\BB}=\CatS\otimes\Univ[\BB]$ of (small) $\BB$-categories.
	\end{enumerate}
\end{example}

Note that by combining Examples~(3) and~(4) from the above list, we may in particular define a (large) $\BB$-category $\IPSh(\I{C})=\iFun(\I{C}^\op,\Univ)$ for every $\BB$-category $\I{C}$ (where $\I{C}^\op$ is defined via $\I{C}^\op(A)=\I{C}(A)^\op$ for every $A\in\BB$). We refer to $\IPSh(\I{C})$ as the $\BB$-category of \emph{presheaves} on $\I{C}$.

\subsection{Left fibrations and Yoneda's lemma}
The notion of left fibrations of $\infty$-categories admits a straightforward $\BB$-categorical analogue:
\begin{definition}
	A functor $p\colon \I{P}\to\I{C}$ of $\BB$-categories is a  \emph{left fibration} if $p(A)$ is a left fibration of $\infty$-categories for every $A\in\BB$. A functor $j\colon \I{I}\to\I{J}$ is \emph{initial} if it is left orthogonal to every left fibration, i.e.\ if for every left fibration $p\colon \I{P}\to\I{C}$ the commutative square
	\begin{equation*}
		\begin{tikzcd}
			\map{\Cat(\BB)}(\I{J}, \I{P}) \arrow[r, "j^\ast"]\arrow[d, "p_\ast"]& \map{\Cat(\BB)}(\I{I}, \I{P})\arrow[d, "p_\ast"]\\
			\map{\Cat(\BB)}(\I{J}, \I{C}) \arrow[r, "j^\ast"]& \map{\Cat(\BB)}(\I{I}, \I{C})
		\end{tikzcd}
	\end{equation*}
	is a pullback.
	
	Dually, $p$ is said to be a \emph{right fibration} if $p(A)$ is a right fibration for every $A\in\BB$., and $j$ is \emph{final} if it is left orthogonal to every right fibration.
\end{definition}
The significance of right and left fibrations is rooted in the fact that they are classified by functors into the universe $\Univ$ of $\BB$. More precisely, let $\LFib(\I{C})$ be the full subcategory of $\Over{\Cat(\BB)}{\I{C}}$ that is spanned by the left fibrations. If $f\colon \I{C}\to\I{D}$ is a functor, then one obtains a pullback map $f^\ast\colon \LFib(\I{D})\to\LFib(\I{C})$, so that one may regard $\LFib(-)$ as a functor $\Cat(\BB)^\op\to\CatSS$. Dually, one can define a functor $\RFib(-)\colon \Cat(\BB)^\op\to\CatSS$. One then obtains the following $\BB$-categorical version of the \emph{straightening equivalence}:
\begin{proposition}[{\cite[\S~4.5]{Yoneda}}]
	\label{prop:straightening}
	There is an equivalence of $\infty$-categories
	\begin{equation*}
		\Fun_{\BB}(\I{C},\Univ[\BB])\simeq\LFib(\I{C})
	\end{equation*}
	which is natural in $\I{C}\in\Cat(\BB)$. Dually, one has an equivalence
	\begin{equation*}
		\PSh[\BB](\I{C})\simeq\RFib(\I{C})
	\end{equation*}
	that is natural in $\I{C}\in\Cat(\BB)$.
\end{proposition}
The straightening equivalence can be  understood quite explicitly. In fact, let $1_{\Univ}\colon 1_{\BB}\to \Univ$ be the map that encodes the final object $1_{\BB}\in \BB$ (via the equivalence $\Fun_{\BB}(1_{\BB},\Univ)\simeq\Gamma_{\BB}\Univ\simeq\BB$), and let $(\pi_{1_{\Univ}})_!\colon\Under{\Univ}{1_{\Univ}}\to\Univ$ be the pullback of $(d_1,d_0)\colon\iFun(\Delta^1,\Univ)\to\Univ\times\Univ$ along the functor $1_{\BB}\times\id\colon\Univ\to\Univ\times\Univ$. Then $(\pi_{1_{\Univ}})_!$ is the  \emph{universal} left fibration, in the following sense:
\begin{proposition}[{\cite[\S~4.6]{Yoneda}}]
	\label{prop:universalLeftFibration}
	For every functor $F\colon \I{C}\to\Univ$, the left fibration $\Under{\I{C}}{F}\to\I{C}$ that corresponds to $F$ via the straightening equivalence is given by the pullback of $(\pi_{1_{\Univ}})_!$ along $F$.
\end{proposition}

As a consequence of Proposition~\ref{prop:universalLeftFibration}, we may construct internal mapping functors as follows: recall that if $\CC$ is an $\infty$-category, there is an $\infty$-category $\Tw(\CC)$, called the \emph{twisted arrow $\infty$-category}, together with a left fibration $p_{\CC} \colon \Tw(\CC)\to\CC^{\op}\times\CC$. As a complete Segal space, the twisted arrow $\infty$-category can be defined via $\Tw(\CC)_n=\map{\CatS}((\Delta^{n})^\op\diamond\Delta^n,\CC)$, where $-\diamond -$ denotes the join of $\infty$-categories, and the left fibration $p_{\CC}$ is simply induced by precomposition with the two inclusions $(\Delta^n)^\op\into(\Delta^n)^{\op}\diamond\Delta^n\hookleftarrow\Delta^n$. As these constructions are evidently natural in $\CC$, we may thus define:
\begin{definition}
	For every $\BB$-category $\I{C}$, one defines the  \emph{twisted arrow} $\BB$-category $\Tw(\I{C})$ via $\Tw(\I{C})(A)=\Tw(\I{C}(A))$, and one defines a left fibration $p_{\I{C}}\colon \Tw(\I{C})\to\I{C}^{\op}\times\I{C}$ via $p_{\I{C}}(A)=p_{\I{C}(A)}$. We denote by $\map{\I{C}}(-,-)\colon\I{C}^\op\times\I{C}\to\Univ$ the functor which classifies the left fibration $p_{\I{C}}$.
\end{definition}
In other words, the internal mapping functor $\map{\I{C}}(-,-)$ is uniquely determined by the property that pulling back the universal left fibration along $\map{\I{C}}(-,-)$ recovers $p_{\I{C}}$. Note that by transposing $\map{\I{C}}(-,-)$, one obtains a functor $h_{\I{C}}\colon \I{C}\to\IPSh(\I{C})$. We refer to the latter as the \emph{Yoneda embedding}. This functor satisfies the following $\BB$-categorical version of Yoneda's lemma:
\begin{proposition}[{\cite[Theorem~4.7.8]{Yoneda}}]
	For every $\BB$-category $\I{C}$, the composition
	\begin{equation*}
		\I{C}^{\op}\times\IPSh(\I{C})\xrightarrow{h_{\I{C}}\times\id}\IPSh(\I{C})^\op\times\IPSh(\I{C})\xrightarrow{\map{\IPSh(\I{C})}(-,-)}\Univ
	\end{equation*}
	is equivalent to the evaluation map $\ev\colon \I{C}^{\op}\times\IPSh(\I{C})\to\Univ$. In particular, the Yoneda embedding $h_{\I{C}}$ is (section-wise) fully faithful.
\end{proposition}

\subsection{Adjunctions and colimits}
It follows from Example~\ref{ex:MainList} that the $ \infty $-category $ \Cat(\BB) $ is enriched in $ \CatS $. Thus we may make use of the $2$-categorical notion of adjunctions to define adjoint functors between $\BB$-categories. Alternatively, one could also make use of internal mapping functors to make sense of a notion of adjunctions of $\BB$-categories. These two approaches turn out to be equivalent:
\begin{proposition}[{\cite[Proposition~3.3.4]{Colimits}}]
	For a pair $l\colon \I{C}\to\I{D}$ and $r\colon \I{D}\to\I{C}$ of functors between $\BB$-categories, the following are equivalent:
	\begin{enumerate}
	\item $l$ is left adjoint to $r$ in the $2$-categorical sense;
	\item there is an equivalence $\map{\I{D}}(l(-),-)\simeq\map{\I{C}}(-,r(-))$ of functors $\I{D}^{\op}\times\I{C}\to\Univ$.
	\end{enumerate}
\end{proposition}
This notion of adjunctions can be characterised in the sheaf-theoretic language as follows:
\begin{proposition}[{\cite[Proposition~3.2.9]{Colimits}}]
	\label{prop:characterisationAdjunctions}
	A functor of $ \BB $-categories $ f \colon \I{C} \to \I{D} $ admits a left adjoint if and only if (1) for every $A\in\BB$ the functor $ f(A) $ admits a left adjoint $ g_A $,  and (2) for every map $ s \colon B \to A $ in $\BB$ the horizontal mate of the commutative square
	\[\begin{tikzcd}
		{\I{C}(A)} & {\I{D}(A)} \\
		{\I{C}(B)} & {\I{D}(B)}
		\arrow["{s^*}", from=1-2, to=2-2]
		\arrow["{s^*}"', from=1-1, to=2-1]
		\arrow["{f(B)}", from=2-1, to=2-2]
		\arrow["{f(A)}", from=1-1, to=1-2]
	\end{tikzcd}\]
	is invertible. The dual statement of right adjoints holds as well.
\end{proposition}
Using the notion of adjunctions of $\BB$-categories, we can now define internal limits and colimits as follows:
\begin{definition}[{\cite[\S~4.1 and~4.2]{Colimits}}]
	Let $ \I{I} $ and $ \I{C} $ be $ \BB $-categories.
	We will say that $ \I{C} $ has $ \I{I} $-indexed colimits if the functor
	\[
		\diag \colon \I{C} \to \iFun(\I{I},\I{C})
	\]
	induced by precomposition with the unique map $ \I{I} \to 1_{\BB}$ admits a left adjoint, in which case we denote this left adjoint by $ \colim_{\I{I}} $.
	Similarly, a functor $ \I{C} \to \I{D} $ between $ \BB $-categories with $ \I{I} $-indexed colimits is said to \emph{preserve $ \I{I} $-indexed colimits}, if the canonical lax square
	\[\begin{tikzcd}
		{\iFun(\I{I},\I{C})} & {\I{C}} \\
		{\iFun(\I{I},\I{D})} & {\I{D}}
		\arrow["{\colim_{\I{I}}}", from=1-1, to=1-2]
		\arrow["f", from=1-2, to=2-2]
		\arrow["{f_*}"', from=1-1, to=2-1]
		\arrow["{\colim_{\I{I}}}"', from=2-1, to=2-2]
		\arrow[Rightarrow, from=2-1, to=1-2, shorten=4mm]
	\end{tikzcd}\]
	commutes. One defines $\I{I}$-indexed limits and the property of a functor to preserve these in the evident dual way.
\end{definition}

\begin{example}[{\cite[Examples~4.1.13,~4.1.14,~4.2.6 and~4.2.7]{Colimits}}]
	\label{ex:TwoKindsOfBColimits}
	There are two important special instances of the above definition:
	\begin{enumerate}
		\item If $ \I{I} = \const_{\BB} (\II) $ for an $ \infty $-category $ \II $ we have an equivalence $ \iFun(\I{I},\I{C})(A) \simeq \Fun(\II,\I{C}(A)) $ functorially in $A\in\BB$.
		Using Proposition~\ref{prop:characterisationAdjunctions} it follows that $ \I{C} $ has $ \I{I}$-indexed colimits if and only if (1) the $ \infty $-category $ \I{C}(A) $ has $ \II $-indexed colimits for every $ A \in \BB $ and (2) for every map $s\colon B\to A$ in $\BB$ the transition functor $ s^* \colon \I{C}(A) \to \I{C}(B) $ preserves $ \II $-indexed colimits. In particular, the colimit of a diagram $\I{I}\to \I{C}$ is simply the colimit of the transposed diagram $\II\to \Gamma_{\BB}(\I{I})$ in the $\infty$-categorical sense. Similarly, a functor $f\colon \I{C}\to\I{D}$ preserves $\I{I}$-indexed colimits if and only if $f(A)$ preserves $\II$-indexed colimits for all $A\in\BB$.
		\item If $ \I{I} = G$ for some $ G \in \BB $ we have an equivalence $ [\I{I},\I{C}])(A) \simeq \I{C}(G \times A) $ functorially in $A\in\BB$.
		Proposition~\ref{prop:characterisationAdjunctions}  thus implies that that $ \I{C} $ has $\I{I}$-indexed colimits if and only if for every $A\in\BB$ the transition functor $ \pi_G^* \colon \I{C}(A) \to \I{C}(G \times A) $ has a left adjoint $ (\pi_G)_!$ that commutes with the transition functors $ s^* \colon \I{C}(A) \to \I{C}(B) $ for every map $ s \colon B \to A $ in $ \BB $. Similarly, a functor $f\colon \I{C}\to\I{D}$ preserves $\I{I}$-indexed colimits if for every $A\in\BB$ the natural map $(\pi_G)_!f(G\times A)\to f(A)(\pi_G)_!$ is an equivalence.
	\end{enumerate}
\end{example}

\subsection{Cocompleteness and cocontinuity}
In higher category one often considers $\infty$-categories that have colimits indexed by all $\infty$-categories $\II$ that are contained in some full subcategory $ \UU \subset \CatS$.
Such $\infty$-categories are typically called $\UU$-cocomplete.
For example, if $ \UU = \CatS $ such $ \infty $-categories will simply be called cocomplete.
One might therefore expect that it is reasonable to call a $ \BB $-category $ \I{C} $ cocomplete if it has colimits indexed by any $ \BB $-category $ \I{I} $.
However, it turns out that this notion is too weak and a stronger notion of cocompleteness has to be imposed, since otherwise expected properties, such as the adjoint functor theorem, will not hold.
The main difference is that one also has to take diagrams into account that are only \emph{locally} defined, i.e.\ not on all of $\BB$ but only on a slice $ \BB_{/A} $. More precisely, instead of just considering diagrams $\I{I}\to\I{C}$ where $\I{I}$ is a $\BB$-category, one also has to take into account diagrams of the form $\I{I}\to\pi_A^\ast\I{C}$ where $\I{I}$ is a $\Over{\BB}{A}$-category, for arbitrary $A\in\BB$.
This leads to the following definition:

\begin{definition}[{\cite[\S~5.2]{Colimits}}]
	Let $ \I{U} $ be an \emph{internal class of $\BB$-categories}, by which we mean a full subsheaf $\I{U}\subset \ICat_{\BB}$.
	We call a $ \BB $-category \emph{$ \I{U} $-cocomplete} if for every $ A \in \BB $ the $ \BB_{/A} $-category $ \pi_A^* (\I{C}) $ admits all colimits indexed by any $ \I{I} \in \I{U}(A) $.
	Similiarly a functor $ f \colon \I{C} \to \I{D} $ is called \emph{$ \I{U} $-cocomplete} if for every $ A \in \BB$, the functor $ \pi_A^* f $ preserves $ \I{I} $-indexed colimits for any $ \I{I} \in \I{U}(A) $.
	We will call $ \I{C} $ \emph{cocomplete} if it is $ \ICat_\BB $-cocomplete, and $f$ \emph{cocontinuous} if it is $\ICat_{\BB}$-cocontinuous.
\end{definition}

As usual we have dual notions of completeness and continuity.

\begin{example}
    \label{ex:twoExOfInternalClasses}
        We will later need the following internal classes:
	\begin{enumerate}
		\item The internal class of \emph{finite} $\BB$-categories $ \IFin_\BB \subset \ICat_\BB$ is the subsheaf spanned by the locally constant sheaves of finite $ \infty $-categories, i.e. by those $\Over{\BB}{A}$-categories $\I{I}$ (where $A\in\BB$ is arbitrary) for which there is a cover $(s_i)\colon\bigsqcup_i A_i\onto A$ in $\BB$ and an equivalence $s_i^\ast(\I{I})\simeq \const_{\Over{\BB}{A_i}}(\II_i)$ of $\Over{\BB}{A_i}$-categories for each $i$, where $\II_i$ is a \emph{finite} $\infty$-category (cf.~\cite[\S~2.2.3]{PresTop}).
		With this definition, a $ \BB $-category $ \I{C} $ is $ \IFin_\BB $-complete if and only if it takes values in the $\infty$-category $\CatS^{\lex}$ of left exact $\infty$-categories and left exact functors~\cite[Proposition~2.2.3.5]{PresTop}.
		We say that such $\BB$-categories have \emph{finite limits } (or are \emph{left exact}).
		Similiarly a functor $ f \colon \I{C} \to \I{D} $ is $ \Fin_\BB $-continuous if and only if $ f(A) $ preserves finite limits for every $ A \in \BB $~\cite[Proposition~2.2.3.5]{PresTop}.
		We say that such functors \emph{preserve finite limits} (or are \emph{left exact}).
		\item For any $ A \in \BB $ let $ \IFilt(A) \subset \Cat(\BB_{/A}) $ be the full subcategory spanned by those $\Over{\BB}{A}$-categories $ \I{I} $ for which the functor of $\Over{\BB}{A}$-categories $ \colim_{\I{I}} \colon [\I{I},\Univ[{\BB_{/A}}]] \to \Univ[\BB_{/A}] $ is left exact.
		Then $ \IFilt $ defines an internal class \cite[Remark 2.1.1.2]{PresTop}.
		We call the objects in $ \IFilt $ \emph{filtered} $ \BB $-categories.
		We will say that a functor \emph{preserves filtered colimits} if it is $ \IFilt $-cocontinuous.
	\end{enumerate}
\end{example}

For a small $ \BB $-category $ \I{C} $, we write $ \IPSh(\I{C}) = \iFun(\I{C}^\op,\Univ[\BB]) $ for the $\BB$-category of \emph{presheaves} on $\I{C}$.
As for usual $ \infty $-categories, we have that $ \IPSh(\I{C}) $ is the \emph{free cocompletion} of $ \I{C} $ in the following sense:

\begin{theorem}[{\cite[Theorem 7.1.1]{Colimits}}]\label{thm:FreeCocompletions}
	Let $\I{E} $ be a cocomplete $ \BB $-category and $ \I{C} $ a small $ \BB $-category.
	Then precomposition with the Yoneda embedding $h_{\I{C}}\colon \I{C} \into \IPSh(\I{C})$ induces an equivalence
	\[
		\Fun_\BB^{\cc}(\IPSh(\I{C}),\Univ[\BB]) \xrightarrow{\simeq }\Fun_\BB(\I{C},\I{E}).
	\]
	Here $ \Fun_\BB^{\cc}(\I{C},\Univ[\BB])\subset \Fun_{\BB}(\I{C},\Univ[\BB])$ denotes the full subcategory spanned by the cocontinuous functors.
\end{theorem}

\subsection{Presentable $ \BB $-categories and $ \BB $-topoi}
Lastly, we turn to the $\BB$-categorical analogues of presentable $\infty$-categories and $\infty$-topoi:
\begin{definition}[{\cite[\S~2.4.2]{PresTop}}]
	A $ \BB $-category $ \I{C} $ is presentable if it is cocomplete and for any $ A \in \BB $ the $ \infty $-category $ \I{C}(A) $ is presentable. We denote by $\LPr(\BB)\into\Cat(\BB)$ the (non-full) subcategory that is spanned by the presentable $\BB$-categories and the cocontinuous functors between them.
\end{definition}

\begin{example}
	The $ \infty $-category $ \Univ[\BB] $ is presentable.
	More generally, for any small $\BB$-category $\I{C}$ the associated presheaf $\BB$-category $\IPSh(\I{C})$ is presentable.
\end{example}

\begin{remark}
	There is also a more intrinsic characterisation of $ \BB $-categories à la Lurie, Simpson \cite[Theorem 2.4.2.5]{PresTop}.
	Furthermore, presentable $ \BB $-categories have many of the expected properties. For example, the usual adjoint functor theorems hold \cite[Proposition 2.4.3.1 \& Proposition 2.4.3.3]{PresTop}.
\end{remark}

\begin{definition}[{\cite[Theorem~3.2.3.1]{PresTop}}]
	A $ \BB $-category $ \I{X} $ is called a \emph{$ \BB $-topos} if there is a small $ \BB $-category $ \I{C} $ and an adjunction
	\[
	\begin{tikzcd}[column sep={6em,between origins}]
	\IPSh(\I{C}) \arrow[r, bend left=30, shift left=4pt, "L"{name=U}, start anchor=east, end anchor=west] \arrow[from=r, bend left=30, shift left=4pt, "i"{name=D}, start anchor=west, end anchor=east]\arrow[from=U, to=D, phantom, "\bot"]& \I{X}
	\end{tikzcd}
	\]
	where $ L $ is left exact and $ i $ is (section-wise) accessible and fully faithful. An \emph{algebraic morphism} of $\BB$-topoi is a left exact left adjoint functor.
	A \emph{geometric morphism} of $ \BB $-topoi is a right adjoint functor whose left adjoint is an algebraic morphism.
	We denote the subcategory of $ \Cat(\BB) $ spanned by the $ \BB $-topoi and algebraic morphisms by $ \LTop(\BB) $ and the subcategory spanned by the $\BB$-topoi and geometric morphisms by $\RTop(\BB)$.
\end{definition}

\begin{remark}
	One can also characterise $ \BB $-topoi as presentable $ \BB $-categories that satisfy \emph{descent} (see \cite[Defintion 3.1.1.4]{PresTop}).
	Furthermore there is also a characterisation in terms of an analogoue of Giraud's axioms \cite[Theorem 3.2.1.5]{PresTop}.
\end{remark}

It is not hard to see that the $ \infty $-category $ \PSh[\BB](\I{C}) $ is an $ \infty $-topos and therefore also $ \Gamma_{\BB} \I{X} $ is an $ \infty $-topos for any $ \BB $-topos $ \I{X} $. Moreover, if $f_\ast\colon \I{X}\to\I{Y}$ is a geometric morphism of $\BB$-topoi, then $\Gamma_{\BB}(f)$ is a geometric morphism of $\infty$-topoi. Thus, the functor $\Gamma_{\BB}$ restricts to a map $\RTop(\BB)\to\RTopS$. Now it is a consequence of Theorem~\ref{thm:FreeCocompletions} that the universe $ \Univ[\BB] $ is the final object of $\RTop(\BB)$~\cite[Corollary~3.2.2.4]{PresTop}.
Hence one obtains an induced functor $ \Gamma_{\BB} \colon \RTop(\BB) \to (\RTopS)_{/ \BB} $.

\begin{theorem}[{\cite[Theorem~3.2.5.1]{PresTop}}]\label{thm:TopoiAreRelativeTopoi}
	The functor $ \Gamma_\BB \colon \RTop(\BB) \to \Over{(\RTopS)}{\BB} $ is an equivalence.
\end{theorem}

The inverse of this equivalence is also easy to describe:
given a geometric morphism $ f_* \colon  \XX \to \BB $, the associated $\BB$-topos is given by $ f_* \Univ[\XX] $, the image of the universe of $ \XX $ along $ f_*\colon \Cat(\XX)\to\Cat(\BB) $. Explicitly, this sheaf is given by $\Over{\XX}{f^\ast(-)}$.
We will constantly use the above theorem to switch back and forth between $\BB$-topoi and $\infty$-topoi over $\BB$.

\section{Proper geometric morphisms}
\label{sec:ProperMorphisms}

In this chapter, we study the relationship between proper and compact morphisms of $\infty$-topoi. Most of the chapter is devoted to the proof of Theorem~\ref{thm:MainTheorem}, which states that a geometric morphism $\XX\to\BB$ is proper precisely if it is compact. The notion of properness of a geometric morphism of $\infty$-topoi is due to Lurie and can be defined as follows:
\begin{definition}[{\cite[Definition~7.3.1.4]{htt}}]
	\label{def:Proper}
	Let $p_\ast\colon \XX\to \BB$ be a geometric morphism of $\infty$-topoi. We say that $f_\ast$ is \emph{proper} if for every commutative diagram
\[\begin{tikzcd}
	{\YY'} & \YY & \XX \\
	{\AA'} & \AA & \BB
	\arrow["{g_*}", from=1-2, to=1-3]
	\arrow["{p_*}", from=1-3, to=2-3]
	\arrow["{q_*}"', from=1-2, to=2-2]
	\arrow["{f'_*}"', from=2-1, to=2-2]
	\arrow["{q'_*}"', from=1-1, to=2-1]
	\arrow["{g'_*}", from=1-1, to=1-2]
	\arrow["{f_*}"', from=2-2, to=2-3]
\end{tikzcd}\]
	in $\RTopS$ in which both squares are cartesian, the left square is left adjointable, in the sense that the mate transformation $(f^\prime)^\ast q_\ast\to q^\prime_\ast(g^\prime)^\ast$ is an equivalence. 
\end{definition}

\begin{example}
    \label{ex:closedimmersions}
	It follows from \cite[Proposition 7.3.12 and Corollary 7.3.2.13]{htt} that any \emph{closed immersion} of $ \infty $-topoi, in the sense of  \cite[Definition 7.3.2.6]{htt}, is proper.
\end{example}

\begin{example}
    Let $ p \colon Y \to X$ be a proper and separated morphism of topological spaces.
    In \S~\ref{sec:Topology} we will prove that then the geometric morphism $ p_* \colon \Shv(Y) \to \Shv(X) $ is proper, generalising a result of Lurie~\cite[Theorem 7.3.1.16]{htt}.
\end{example}
Before we can prove Theorem~\ref{thm:MainTheorem}, we first need to make precise what we mean by a \emph{compact} geometric morphism. We do so in \S~\ref{sec:compactBTopoi}. In \S~\ref{sec:toposicCone} and~\S~\ref{sec:pullbacksLocalisations}, we discuss two auxiliary steps that are required for the proof of Theorem~\ref{thm:MainTheorem}: the $\infty$-toposic cone construction and the compatibility of pullbacks with localisations of $\infty$-topoi. Lastly, we put everything together in \S~\ref{sec:proof} to finish the proof.

In \S~\ref{sec:Topology} we discuss how we can apply Theorem~\ref{thm:MainTheorem} to show Theorem~\ref{thm:ProperMapsTopology}, which says that the geometric morphism associated with a proper and separated map of topological spaces is proper. Finally, in~\ref{sec:coefficients} we discuss a variant of Theorem~\ref{thm:MainTheorem} in which we allow coefficients in an arbitrary compactly generated $\infty$-category.

\subsection{Compact $\BB$-topoi}
\label{sec:compactBTopoi}
Recall that an $\infty$-topos $\XX$ is said to be \emph{compact} if the global sections functor $\Gamma_{\XX}\colon\XX\to\SS$ preserves filtered colimits. In this sections, we study the $\BB$-toposic analogue of this notion. To that end, observe that as $\Univ[\BB]$ is the final $\BB$-topos, there is a unique geometric morphism $\Gamma_{\I{X}}\colon\I{X}\to\Univ[\BB]$ for every $\BB$-topos $\I{X}$, which we refer to as the \emph{global sections functor}. We may now define:
\begin{definition}
	\label{def:compactBTopos}
	A $\BB$-topos $\I{X}$ is said to be \emph{compact} if the global sections functor $\Gamma_{\I{X}}\colon\I{X}\to\Univ[\BB]$ preserves filtered colimits, i.e.\  is $\IFilt$-cocontinuous. We say that a geometric morphism $p_* \colon \XX \to \BB$ is compact if the associated $\BB$-topos $p_\ast\Univ[\XX]$ is compact. 
\end{definition}

\begin{example}
	\label{ex:finiteGroupoidCompact}
	If $A\in\BB$ is an arbitrary object, the \'etale geometric morphism $ (\pi_A)_\ast\colon\BB_{/A} \to \BB $ is compact if and only if $ A $ is \emph{internally compact} in $ \BB $, i.e.\ if the functor $\map{\Univ}(A,-)\colon\Univ\to\Univ$ preserves filtered colimits. To see this, first note that $(\pi_A)_\ast$ corresponds to the \'etale $\BB$-topos $\iFun(A,\Univ[\BB])$ (see~\cite[\S~3.2.9]{PresTop}) and the unique geometric morphism into $\Univ[\BB]$ is given by the limit functor $\lim_A\colon\iFun(A,\Univ[\BB])\to\Univ[\BB]$ (as its left adjoint $\diag_A$ is again a right adjoint and therefore preserves all limits~\cite[Proposition~5.2.5]{Colimits}). Moreover, since $A\simeq\colim_A \diag_A 1_{\Univ}$, where $1_{\Univ}\colon 1_{\BB}\to\Univ$ is the section endoding the final object $1_{\BB}\in\BB$ (see~\cite[Proposition~4.4.1]{Colimits}), the adjunctions $\colim_A\dashv \diag_A\dashv \lim_A$ imply that we obtain an identification $\map{\Univ}(A,-)\simeq\map{\Univ}(1_{\Univ}, \lim_A\diag_A)\simeq \lim_A\diag_A$ (since $\map{\Univ}(1_{\Univ},-)$ is equivalent to the identity, see~\cite[Proposition~4.6.3]{Yoneda}). Hence $\lim_A$ preserving filtered colimits implies that $A$ is internally compact. To see the converse, note that by~\cite[Corollary~2.2.3.8]{PresTop}, $A$ being internally compact is equivalent to $ A $ being locally constant with compact values. If this is the case, then the fact that $\IFilt$-cocontinuity can be checked locally in $\BB$ (see~\cite[Remark~5.2.3]{Colimits}) allows us to reduce to the case where $A$ is constant with compact value. In other words, $A$ is a retract of a finite $\BB$-groupoid, so that we may further reduce to the case where $A$ is already finite. In this case, $\lim_A$ preserves filtered colimits by the very definition of filteredness. 
\end{example}

\begin{warning}
    \label{warning:CompactIsNotExternalColims}
	In the context of Definition~\ref{def:compactBTopos}, it is essential that we require $\Gamma_{\I{X}}\colon\I{X}\to\Univ[\BB]$ to be $\IFilt$-cocontinuous instead of just asking for $p_\ast$ to preserve filtered colimits. In fact, if $A\in\BB$ is an arbitrary object, we saw in Example~\ref{ex:finiteGroupoidCompact} that $(\pi_A)_\ast\colon\Over{\BB}{A}\to\BB$ is compact if and only if $A$ is internally compact. On the other hand, $(\pi_A)_\ast$ preserves filtered colimits if and only if the functor $\ihom_{\BB}(A,-)\colon \BB\to\BB$ preserves filtered colimits (where $\ihom_{\BB}$ denotes the internal hom of $\BB$). By~\cite[Proposition~4.4.11]{Colimits}, $A$ being internally compact implies that $\ihom_{\BB}(A,-)$ preserves filtered colimits, but the converse is not true in general. For example, if $X$ is a coherent space, then any quasi-compact open $U\subset X$ defines an object in the $\infty$-topos $\Shv(X)$ satisfying the latter condition (since quasi-compact opens in $X$ define compact objects in $\Shv(X)$ and generate this $\infty$-topos under colimits). On the other hand, $U$ is in general quite far from being locally constant and can therefore not be internally compact.
\end{warning}

\begin{remark}
    \label{rem:Spec(R)Counterexample}
	Let $ \XX $ be a $ 1 $-localic $ \infty $-topos. 
	If $ \XX $ is compact, the associated $ 1 $-topos $ \Disc(\XX) $ of $0$-truncated objects in $\XX$ is \emph{tidy} in the sense of~\cite{Moerdijk2000}.
	However, the converse it not true in general.
	For example, any coherent $ 1 $-topos is tidy, but $ 1 $-localic coherent $ \infty $-topoi are not compact in general.
	An explicit counterexample is $ \operatorname{Spec}(\mathbb{R})_{\et}\simeq\Fun ( B (\mathbb{Z}/ 2\mathbb Z),\SS) $, which cannot be tidy since $  B (\mathbb{Z}/ 2\mathbb Z) \simeq \mathbb{RP}^\infty $ is not compact in $ \SS $.
\end{remark}

\begin{example}
    \label{ex:cdstructurecompact}
	Any $ \infty $-topos of the form $ \Shv^{\tau}(\CC) $ where $ \CC $ is an $ \infty $-category with an initial and a terminal object and $ \tau $ a topology generated by a cd-structure is compact.
	This follows since under these assumption $ \Shv^{\tau}(\CC) \to \PSh(\CC) $ commutes with filtered colimits and $ \PSh(\CC) $ is always compact if $ \CC $ has a terminal object.
	Example of such topologies from algebraic geometry include the Zariski-, Nisnevich- and cdh-topology.
\end{example}

\begin{lemma}
	\label{lem:properclosedunderretracs}
	Compact $ \BB $-topoi are stable under retracts in $ \RTop(\BB) $.
\end{lemma}
\begin{proof}
	Let $ \I{X} $ be a compact $ \BB $-topos and
	\[
		\I{Y} \xrightarrow{s_*} \I{X} \xrightarrow{r_*} \I{Y}
	\]
	a retract diagram of $ \BB $-topoi.
	We thus get a retract diagram of functors
	\[
		\Gamma_{\I{Y}} \simeq \map{\I{Y}}(1,-) \xrightarrow{r^*} \map{\I{X}} (r^* 1, r^*(-)) \xrightarrow{s^*} \map{\I{Y}}(s^*r^*1,s^*r^*(-)) \simeq \Gamma_{\I{Y}}.
	\]
	Furthermore $  \map{\I{X}} (r^* 1, r^*(-)) \simeq \Gamma_{\I{X}} \circ r^*$ and thus $ \Gamma_{\I{Y}}$ preserves filtered colimits as a retract of a filtered colimit preserving functor.
\end{proof}

\begin{example}
	\label{ex:stably_compact_locales}
	The theory of $\BB$-locales, that we developed in \cite[\S~3.4]{PresTop}, can be used to give a large class of examples of compact $\BB$-topoi.
    For this recall that a $\BB$-locale is \emph{stably compact} if it is a retract of a coherent $ \BB $-locale, see \cite[Definition~3.4.7.1]{PresTop}. 
	We claim that the $\BB$-topos $ \IShv(\I{L}) $  of \emph{sheaves on} $\I{L}$, as defined in \cite[\S~3.4.4]{PresTop}, is a compact $ \BB $-topos.
	Indeed, using Lemma~\ref{lem:properclosedunderretracs} it suffices to see that $ \IShv(\I{L}) $ is compact, whenever $\I{L}$ is coherent.
	In this case recall from \cite[Proposition~3.4.7.8]{PresTop} that there is an equivalence $ \IShv(\I{L}) \simeq \IShv^{\fin}(\I{L}^{\compact}) $ that identifies sheaves on $\I{L}$ with \emph{finitary} sheaves on the internally compact objects of $ \I{L} $, see also \cite[Definition~3.4.7.3]{PresTop}.
	By \cite[Lemma~3.4.7.7]{PresTop} the embedding
	\[
		\IShv^{\fin}(\I{L}^{\compact}) \hookrightarrow \IPSh(\I{L}^{\compact})
	\]
	preserves filtered colimits.
	Thus it suffices to see that $ \Gamma_{\IPSh(\I{L}^{\compact})} $ preserves filtered colimits.
	Now note that because $ \I{L}^{\compact} $ has a final object $ 1 $, it follows that $ \Gamma_{\IPSh(\I{L}^{\compact})} \simeq \ev_1$ by \cite[Proposition~4.6.1]{Colimits}.
	In particular it preserves filtered colimits.
\end{example}

We are now ready to state the main result of this paper:
\begin{theorem}
	\label{thm:CompactProperBaseChange}
	A geometric morphism $p_* \colon \XX \to \BB $ is proper if and only if it is compact.
\end{theorem}

The full proof of Theorem~\ref{thm:CompactProperBaseChange} is rather involved and will be given in \S~\ref{sec:proof}. One implication, however, is straightforward:
\begin{lemma}
	\label{lem:properImpliesCompact}
	Let $ p_* \colon \XX \to \BB $ be a proper geometric morphism. Then $ p_* $ is compact.
\end{lemma}
\begin{proof}
	Let us denote by $\I{X}=p_\ast\Univ[\XX]$ the $\BB$-topos that corresponds to $p_\ast$.
	Note that if $A\in\BB$ is an arbitrary object, the induced morphism $\Over{\XX}{p^\ast A}\to\Over{\BB}{A}$ is proper as well. As this is the geometric morphism which corresponds to the $\Over{\BB}{A}$-topos $\pi_A^\ast\I{X}$, we may (after replacing $\BB$ with $\Over{\BB}{A}$) reduce to the case where we have to show that if $ \I{I} $ is a filtered $ \BB $-category, then $\Gamma_{\I{X}}$ preserves $\I{I}$-filtered colimits.
	By definition of filteredness, the colimit functor $ \colim_{\I{I}} \colon \iFun(\I{I},\Univ[\BB]) \to \Univ_\BB$ is left exact, which implies that $ \diag \colon  \Univ_\BB \to \iFun(\I{I},\Univ[\BB])$ defines a geometric morphism of $ \BB $-topoi.
	We may therefore consider the two pullback squares
	\[\begin{tikzcd}
		{\I{X}} & \iFun(\I{I},\I{X})& {\I{X}} \\
		{\Univ_{\BB}} & \iFun(\I{I},\Univ[\BB])& {\Univ_\BB}
		\arrow[from=1-3, to=2-3, "\Gamma_{\I{X}}"]
		\arrow[from=2-2, to=2-3, "\lim"]
		\arrow["{(\Gamma_{\I{X}})_\ast}"', from=1-2, to=2-2]
		\arrow[from=1-2, to=1-3, "\lim"]
		\arrow["\diag", from=2-1, to=2-2]
		\arrow["{\Gamma_{\I{X}}}"', from=1-1, to=2-1]
		\arrow["\diag", from=1-1, to=1-2]
	\end{tikzcd}\]
	in $\RTop(\BB)$ (see~\cite[Example~3.2.7.5]{PresTop}).
	As for every $A\in\BB$ the geometric morphism $\Gamma_{\I{X}}(A)\colon \Over{\XX}{p^\ast A}\to\Over{\BB}{A}$ is proper, it follows that the mate of the left square is an equivalence.
	This precisely means that $ \Gamma_{\I{X}}$ commutes with $ \I{I} $-indexed colimits, as desired.
\end{proof}

\subsection{The toposic cone}
\label{sec:toposicCone}
Every topological space $X$ admits a closed immersion into a contractible and locally contractible space $Y$, for example by setting $Y=C(X)$, where $C(X)$ is the \emph{cone} of $X$. In this section, we discuss a $\BB$-toposic analogue of this observation.
To that end, if $\I{X}$ is a $\BB$-topos, recall that the \emph{comma $\BB$-category} $\Comma{\I{X}}{\I{X}}{\Univ}$ is defined via the pullback square
\[\begin{tikzcd}
	\Comma{\I{X}}{\I{X}}{\Univ} \arrow[r, "e^\ast"]& \iFun(\Delta^1,\I{X}) \\
	\Univ & {\I{X}}
	\arrow[from=1-1, to=1-2]
	\arrow[from=1-2, to=2-2, "d_0"]
	\arrow["\const_{\I{X}}"', from=2-1, to=2-2]
	\arrow[from=1-1, to=2-1, "j^\ast"]
\end{tikzcd}\]
in $\Cat(\BB)$, where $\const_{\I{X}}$ denotes the unique algebraic morphism $\Univ[\BB]\to\I{X}$, i.e.\ the left adjoint of $\Gamma_{\I{X}}$. By~\cite[Proposition~3.2.6.1]{PresTop}, this is a pullback diagram in $\LTop(\BB)$, so that $\Comma{\I{X}}{\I{X}}{\Univ}$ is a $\BB$-topos and $j^\ast$ and $e^\ast$ are algebraic morphisms.
\begin{definition}
	\label{def:toposicMappingCone}
	For any $\BB$-topos $\I{X}$, we refer to the $\BB$-topos $\Comma{\I{X}}{\I{X}}{\Univ}$ as its \emph{$\BB$-toposic right cone} and denote it by $\I{X}^\triangleright$.
\end{definition}

If $\I{X}$ is a $\BB$-topos, let $i^\ast\colon \I{X}^\triangleright\to\I{X}$ be the algebraic morphism that is obtained by composing the functor $d_1\colon\iFun(\Delta^1,\I{X})\to\I{X}$ with the upper horizontal map in the defining pullback square of $\I{X}^\triangleright$.

\begin{remark}
	\label{rem:mappingConeRecollement}
	Suppose that $\I{X}$ is a $\BB$-topos and let $f_\ast\colon \XX\to\BB$ be the associated geometric morphism of $\infty$-topoi. Then the $\infty$-topos $\Cone(f)=\Gamma_{\BB}(\I{X}^\triangleright)$ recovers the comma $\infty$-category $\Comma{\XX}{\XX}{\BB}$ and is therefore the \emph{recollement} of $\BB$ and $\XX$ along $f^\ast$ in the sense of~\cite[\S~A.8]{Lurie2017}. In particular, $ j_*\colon \BB \to \Cone(f)$ is an open and $ i_* \colon \XX \to \Cone(f) $ a closed immersion of $\infty$-topoi. Note that this in particular implies that $i_\ast\colon \I{X}\to\I{X}^\triangleright$ and $j_\ast\colon\Univ\to\I{X}^\triangleright$ are (section-wise) fully faithful.
\end{remark}

\begin{remark}
	In the situation of Remark~\ref{rem:mappingConeRecollement}, the $\infty$-topos $\Cone(f)$ sits inside a pushout square
	\[\begin{tikzcd}
		\XX & \BB \\
		{\XX \otimes \Delta^1 } & {\Cone(f)}
		\arrow["{f_*}", from=1-1, to=1-2]
		\arrow[from=1-2, to=2-2]
		\arrow["{\id \otimes s^0}"', from=1-1, to=2-1]
		\arrow[from=2-1, to=2-2]
	\end{tikzcd}\]
	in $\RTopS$, where $\XX\otimes\Delta^1$ denotes the tensoring in $\RTopS$ over $\CatS$. Therefore $ \Cone(f) $ is to be thought of as the mapping cone of $f_\ast$.
\end{remark}

Recall from~\cite[Definition~3.3.1.1]{PresTop} that a $\BB$-topos $\I{X}$ is said to be \emph{locally contractible} if $\const_{\I{X}}\colon\Univ[\BB]\to\I{X}$ has a left adjoint $\pi_{\I{X}}$. The following proposition expresses the fact that the $\BB$-toposic right cone $\I{X}^{\triangleright}$ is contractible and locally contractible (in the $\BB$-toposic sense):
\begin{proposition}
	\label{prop:MappingConeFactorisation}
	For every $\BB$-topos $\I{X}$, the $\BB$-topos $\I{X}^\triangleright$ is locally contractible, and the additional left adjoint $\pi_{\I{X}}$ of $\const_{\I{X}^{\triangleright}}$ is equivalent to $j^\ast$. In particular,  $\pi_{\I{X}}$ preserves finite limits.
\end{proposition}
\begin{proof}
	Since $s_0\colon \I{X}\to\iFun(\Delta^1,\I{X})$ is right adjoint to $d_0$ (for example by using~\cite[Proposition~3.1.15]{Colimits}) and by the dual of~\cite[Lemma~6.3.9]{Colimits}, the functor $j_\ast$ is the pullback of $s_0$ along $\epsilon^\ast$. Since $s_0$ is cocontinuous and as $\LPr(\BB)\into\Cat(\BB)$ preserves limits, this implies that $j_\ast$ must be cocontinuous as well and therefore equivalent to $\const_{\I{X}^{\triangleright}}$ (by the universal property of $\Univ$). As this shows that $j^\ast$ is left adjoint to $\const_{\I{X}^{\triangleright}}$, the claim follows.
\end{proof}

\begin{remark}
    For $1$-topoi, the factorization constructed above appears in the proof of \cite[Theorem C.3.3.14]{johnstone2002}.
\end{remark}

\subsection{Compatibility of pullbacks with localisations}\label{sec:pullbacksLocalisations}
The goal of this section is to establish the main technical step towards the proof of Theorem~\ref{thm:CompactProperBaseChange}, the fact that compact geometric morphisms commute with localisations of subtopoi:
\begin{proposition}
	\label{prop:SpecialCaseOfBeckChevalley}
	Consider a pullback square in $ \RTopS $
	\[\begin{tikzcd}
		{\mathcal{Z}^\prime} & {\mathcal{X}} \\
		{\mathcal{Z}} & {\mathcal{B}}
		\arrow["{j^\prime_*}", from=1-1, to=1-2]
		\arrow["p_*", from=1-2, to=2-2]
		\arrow["p^\prime_*", from=1-1, to=2-1]
		\arrow["j_*"', from=2-1, to=2-2]
	\end{tikzcd}\]
	where $ j_* $ is fully faithful and $ p_* $ is compact. 
	Then the mate natural transformation $ p^\prime_* (j^\prime)^* \to j^* p_* $ is an equivalence.
\end{proposition}
Intuitively, Proposition~\ref{prop:SpecialCaseOfBeckChevalley} should hold because the localisation functor $(j^\prime)^\ast\colon \XX \to \ZZ^\prime$ is given by an (internally) filtered colimit.
Indeed, we may pick a bounded local class $ S^{\prime}$ of morphisms in $ \XX $ that is closed under finite limits  in $\Fun(\Delta^1,\XX)$ such that $ (j^\prime)^* $ exhibits $ \ZZ^{\prime}$ as the Bousfield localisation of $ \XX $ at $ S^{\prime}$ (see \cite[Proposition~3.2.10.14]{PresTop}).
We denote by $\iota^\prime \colon \Univ[S^\prime] \into \Univ[\XX]$ the associated full subcategory. Then $S^\prime$ being bounded implies that $\Univ[S^\prime]$ is a small $\XX$-category, and $S^\prime$ being closed under finite limits in $\Fun(\Delta^1, \XX)$ implies that $\Univ[S^\prime]$ has finite limits and is therefore \emph{cofiltered} by \cite[Proposition~2.2.3.7]{PresTop} (i.e.\ its opposite $\Univ[S^\prime]^\op$ is filtered). Now by the formula in \cite[Proposition~3.2.10.14]{PresTop}, we find that for every $X\in\XX$ we obtain a canonical equivalence
\begin{equation*}
	j^\prime_\ast(j^\prime)^\ast(X)\simeq X^{\sh}_{\iota^\prime} = \colim_{\tau < \kappa} T_{\tau}^{\iota^\prime}X
\end{equation*}
where $\kappa$ is a suitably large regular cardinal and $T_{\tau}^{\iota^\prime} X$ is defined recursively by the condition that we have $T_{\tau + 1}^{\iota^\prime} X= \colim_{\Univ[S^\prime]^\op}\map{\Univ[\XX]}(\iota^\prime(-), T_{\tau}^{\iota^\prime} X)$ and that $T_{\tau}^{\iota^\prime}X = \colim_{\tau^\prime < \tau} T_{\tau^\prime}^{\iota^\prime} X$ when $\tau$ is a limit ordinal.
In particular, the endofunctor $j^\prime_\ast(j^\prime)^\ast$ is given by an (iterated) \emph{filtered} colimit. 
So intuitively, $p_\ast$ being compact should imply that this functor carries $j^\prime_\ast(j^\prime)^\ast(X)$ to $j_\ast j^\ast p_\ast(X)$, which precisely means that the mate transformation $p^\prime_\ast (j^\prime)^\ast(X)\to j^\ast p_\ast(X)$ is an equivalence. 
However, we have to be a bit careful at this point: the above formula for $j^\prime_\ast(j^\prime)^\ast$ exhibits $j^\prime_\ast(j^\prime)^\ast(X)$ as a filtered colimit \emph{internal to $\XX$}, whereas $p_\ast$ being compact only implies that this functor commutes with filtered colimits \emph{internal to $\BB$}. Hence, the main challenge is to rewrite the above formula in terms of a filtered colimit internal to $\BB$.

\begin{observation}
	\label{obs:transposingColimits}
	Let $ f_* \colon \XX \to \BB $ be a geometric morphism, $ \I{I} $ a $ \BB $-category and $ \I{C} $ an $ \XX $-category.
	On account of the commutative diagram
	\begin{equation*}
		\begin{tikzcd}
			f_\ast\I{C}\arrow[r, "\diag_{\I{I}}"]\arrow[dr, "f_\ast(\diag_{f^\ast\I{I}})"'] & \iFun[\BB](\I{I}, f_\ast\I{C})\arrow[d, "\simeq"]\\
			& f_\ast\iFun[\XX](f^\ast\I{I},\I{C}),
		\end{tikzcd}
	\end{equation*}
	the $\BB$-category $f_\ast \I{C}$ admits $\I{I}$-indexed colimits if and only if the $\XX$-category $\I{C}$ admits $f^\ast\I{I}$-indexed colimits, and we may identify the colimit functor $\colim_{\I{I}}\colon \iFun(\I{I}, f_\ast \I{C})\to f_\ast\I{C}$ with the composition 
	\[
	\iFun(\I{I}, f_\ast\I{C})\simeq f_\ast \iFun[\XX](f^\ast\I{I},\I{C})\xrightarrow{f_\ast(\colim_{f^\ast\I{I}})} f_\ast\I{C}.
	\] 
	By passing to global sections, this implies that for every diagram $d\colon \I{I}\to f_\ast \I{C}$ with transpose $\bar d\colon f^\ast\I{I}\to \I{C}$, we have a canonical equivalence $\col \BB_{\I{I}} d \xrightarrow{\simeq} \col \XX_{f^\ast\I{I}} \bar{d}$
	in the $\infty$-category $\Gamma_{\XX}(\I{C})=\Gamma_{\BB}(f_\ast \I{C})$ (where the superscripts emphasise internal to which $\infty$-topos the colimits are taken). We will repeatedly use this observation throughout this chapter.
\end{observation}

Suppose now that $S$ is a bounded local class of morphisms in $\BB$ that is closed under finite limits in $\Fun(\Delta^1,\BB)$, and let $\iota\colon\Univ[S]\into \Univ$ be the associated (cofiltered) full subcategory. If $f_\ast\colon\XX\to\BB$ is a geometric morphism, we let $\iota^\prime\colon f^\ast(\Univ[S])\to \Univ[\XX]$ be the functor of $\XX$-categories that arises from transposing $\const_{ f_\ast(\Univ[\XX])}\iota\colon\Univ[S]\to f_\ast(\Univ[\XX])$ across the adjunction $f^\ast\dashv f_\ast$. By \cite[Example~3.2.10.9]{PresTop}, this is a cofiltered $\XX$-category, and the colimit of $\iota^\prime$ is the final object in $\Univ[\XX]$. Therefore, we are in the situation of \cite[Definition~3.2.10.5]{PresTop} and thus obtain an endofunctor $(-)^{\sh}_{\iota^\prime}\colon \Univ[\XX]\to\Univ[\XX]$ via $(-)^{\sh}_{\iota^\prime}=\colim_{\tau < \kappa} T^{\iota^\prime}_{\tau}$, where $\kappa$ is a suitable $\XX$-regular cardinal and where $T^{\iota^\prime}_{\bullet}\colon \kappa\to \Fun_{\XX}(\Univ[\XX], \Univ[\XX])$ is defined via transfinite induction by setting $T_0^{\iota^\prime} = \id$, by defining the map $T_\tau^{\iota^\prime}\to T_{\tau + 1}^{\iota^\prime}$ to be the morphism $\phi\colon T_\tau^{\iota^\prime}\to (T_\tau)^+_{\iota^\prime} = \colim_{f^\ast(\Univ[S])^\op}\map{\Univ[\XX]}(\iota^\prime(-), -)$ from \cite[Remark~3.2.10.4]{PresTop} and finally by setting $T_{\tau}^{\iota^\prime}= \colim_{\tau^\prime < \tau} T_{\tau^\prime}^{\iota^\prime}$ whenever $\tau$ is a limit ordinal. We will slightly abuse notation and also denote by $(-)^{\sh}_{\iota^\prime}$ the underlying endofunctor on $\XX$ that is obtained by passing to global sections. It will always be clear from the context which variant we refer to.

\begin{proposition}
	\label{prop:comparisonofcolimits}
	Consider a pullback square $ \mathcal{Q} $ in $ \RTopS $
	\[\begin{tikzcd}
		{\mathcal{Z}^\prime} & {\mathcal{X}} \\
		{\mathcal{Z}} & {\mathcal{B}}
		\arrow["{j^\prime_*}", from=1-1, to=1-2]
		\arrow["f_*", from=1-2, to=2-2]
		\arrow["{g_*}"', from=1-1, to=2-1]
		\arrow["j_*"', from=2-1, to=2-2]
	\end{tikzcd}\]
	where $ j_* $ (and therefore also $j_\ast^\prime$) is fully faithful.
	Let $S$ be a bounded local class of morphisms, closed under finite limits in $\Fun(\Delta^1,\BB)$, such that $j^*$ is the Bousfield localisation at $S$ (such a local class always exists by \cite[Proposition~3.2.10.14]{PresTop}), and let $\iota\colon \Univ[S]\into\Univ[\BB]$ be the associated full subcategory. Then we obtain an equivalence $j^\prime_\ast (j^\prime)^\ast\simeq (-)^{\sh}_{\iota^\prime}$, where $\iota^\prime\colon f^\ast(\Univ[S])\to \Univ[\XX]$ is the transpose of $\const_{ f_\ast(\Univ[\XX])}\iota$.
\end{proposition}

\begin{remark}
	The above proposition can be thought of as an $\infty$-toposic version of \cite[Theorem C.3.3.14]{johnstone2002}.
\end{remark}

We first prove this proposition in a special case:

\begin{lemma}
	\label{lem:AlphaIsEquForLocWeaklyContr}
	Consider a pullback square $ \mathcal{Q} $ in $ \RTopS $
	\[\begin{tikzcd}
		{\mathcal{Z}^\prime} & {\mathcal{X}} \\
		{\mathcal{Z}} & {\mathcal{B}}
		\arrow["{j^\prime_*}", from=1-1, to=1-2]
		\arrow["h_*", from=1-2, to=2-2]
		\arrow["{h^\prime_*}"', from=1-1, to=2-1]
		\arrow["j_*"', from=2-1, to=2-2]
	\end{tikzcd}\]
	where $ j_* $ is fully faithful and $ h_* $ is locally contractible such that the additional left adjoint $ h_! $ of $ h^* $ preserves finite limits.
	Let $S$ and $\iota$ be as in Proposition \ref{prop:comparisonofcolimits}. 
	Then there is an equivalence $j^\prime_\ast (j^\prime)^\ast\simeq (-)^{\sh}_{\iota^\prime}$, where $\iota^\prime\colon h^\ast(\Univ[S])\to \Univ[\XX]$ is the transpose of $\const_{ h_\ast(\Univ[\XX])}\iota$.
\end{lemma}
\begin{proof}
	By \cite[Proposition~3.3.1.5]{PresTop}, the functor $\iota^\prime$ is fully faithful, and since $h_\ast$ is locally contractible the $\XX$-category $h^\ast(\Univ[S])$ is given by the sheaf $\Univ[S](h_!(-))$. It follows that a map $s\colon X\to Y$ in $\XX$ defines an object of $h^\ast(\Univ[S])(Y)$ if and only if $ h_!(s) \in S $ and the square
	\[
	\begin{tikzcd}
		X & {h^* h_!X} \\
		Y & {h^* h_!Y}
		\arrow[from=2-1, to=2-2]
		\arrow[from=1-2, to=2-2]
		\arrow[from=1-1, to=2-1]
		\arrow[from=1-1, to=1-2]
	\end{tikzcd}
	\]
	is a pullback. Let $W$ be the class of maps in $\XX$ that satisfies these two conditions. Then, since $h_!$ is cocontinuous and preserves finite limits, it easily follows that $W$ is local. Hence we find $h^\ast(\Univ[S])= \Univ[W]$ as full subcategories of $\Univ[\XX]$. Moreover, by the explicit description of $W$, it is clear that $W$ is closed under finite limits in $\Fun(\Delta^1,\XX)$. Thus, by appealing to \cite[Proposition~3.2.10.14]{PresTop}, we only need to verify that $\ZZ^\prime$ is the Bousfield localisation of $\XX$ at $W$.
	We know from \cite[Remark~.3.2.10.15]{PresTop} that $ \ZZ^{\prime}\into \XX $ is obtained as the Bousfield localisation of $\XX$ at the smallest local class $ \overline{h^*S} $ that contains the image $h^\ast S$ of $ S$ along $h^\ast$. Since we clearly have $h^\ast S\subset W$, this immediately implies $W= \overline{f^\ast S}$, hence the claim follows.
\end{proof}

For our next lemma, we need to establish a bit of  notation.
If $p_\ast\colon \XX\to\BB$ is a geometric morphism, then the associated $\BB$-topos $p_\ast(\Univ[\XX])$ is \emph{cartesian closed} (see \cite[Proposition~3.1.3.7]{PresTop}).
This means that there is a functor $\ihom_{p_\ast(\Univ[\XX])}^\BB(-,-)\colon p_\ast(\Univ[\XX])^\op\times p_\ast(\Univ[\XX])\to p_\ast(\Univ[\XX])$ that fits into an equivalence
\begin{equation*}
	\map{p_\ast(\Univ[\XX])}(-\times -, -)\simeq \map{p_\ast(\Univ[\XX])}(-,\ihom_{p_\ast(\Univ[\XX])^\BB}(-,-)).
\end{equation*}
Here the superscript in $\ihom_{p_\ast(\Univ[\XX])}^\BB(-,-)$ indices that we are working internally in $\BB$.
\begin{lemma}
	\label{lem:compactMorphismCommutesWithSheafification}
	Let $p_\ast\colon \XX\to\BB$ be a compact geometric morphism and let $\iota\colon \I{I}\to\Univ$ be a functor where $\I{I}$ is cofiltered and where $\colim \iota$ is the final object. Let $\iota^\prime\colon p^\ast \I{I}\to\Univ[\XX]$ be the transpose of $\const_{ p_\ast(\Univ[\XX])}\iota\colon \I{I}\to p_\ast(\Univ[\XX])$. Then there is an equivalence $p_\ast (-)^{\sh}_{\iota^\prime}\simeq (-)^{\sh}_{\iota}p_\ast$.
\end{lemma}
\begin{proof}
	Since $(-)^{\sh}_{\iota^\prime}$ and $(-)^{\sh}_{\iota}$ are obtained as filtered colimits of iterations of $(-)^{+}_{\iota^\prime}$ and $(-)^{+}_{\iota}$, respectively, and as $p_\ast$ commutes with filtered colimits, it suffices to produce an equivalence $p_\ast (-)^{+}_{\iota^\prime}\simeq (-)^{+}_{\iota}p_\ast$.
	Now for every $X\in\XX$, we have a natural chain of equivalences
	\begin{align*}
		(p_\ast X)^+_{\iota} &= \col {\BB}_{\I{I}^\op}\map{\Univ[\BB]}(\iota(-), \Gamma_{p_\ast(\Univ[\XX])}X)\\
		&\simeq \col {\BB}_{\I{I}^\op}\map{p_\ast(\Univ[\XX])}(\const_{ p_\ast(\Univ[\XX])}\iota(-), X) \\
		&\simeq \col {\BB}_{\I{I}^\op}\Gamma_{p_\ast(\Univ[\XX])}(\ihom_{p_\ast(\Univ[\XX])}^{\BB}(\const_{ p_\ast(\Univ[\XX])}\iota(-), X))\\
		&\simeq \Gamma_{p_\ast(\Univ[\XX])}(\col {\BB}_{\I{I}^\op}\ihom_{p_\ast(\Univ[\BB])}^{\BB}(\const_{ p_\ast(\Univ[\XX])}\iota(-), X))\\
		&\simeq \Gamma_{p_\ast(\Univ[\XX])}(\col {\XX}_{p^\ast\I{I}^\op} \map{p_\ast(\Univ[\XX])}(\iota^\prime(-), X))\\
		&\simeq p_\ast X^{+}_{\iota^\prime}
	\end{align*}
	where the third step follows from \cite[Remark~3.2.10.13]{PresTop}, the fourth step is a consequence of the fact that $\Gamma_{p_\ast(\Univ[\XX])}$ preserves filtered colimits and the fifth step follows from Observation \ref{obs:transposingColimits}. Hence the result follows.
\end{proof}

\begin{proof}[Proof of Proposition~\ref{prop:comparisonofcolimits}]
	Using Proposition~\ref{prop:MappingConeFactorisation}, we may factor the pullback square $ \QQ $ into two squares
	\[\begin{tikzcd}
		{\ZZ^\prime} & \XX \\
		{\ZZ^{\prime\prime}} & \YY \\
		\ZZ & \BB
		\arrow["{j_*}", from=3-1, to=3-2]
		\arrow[from=2-1, to=3-1]
		\arrow["{j^{\prime\prime}_*}"', from=2-1, to=2-2]
		\arrow["{i_*}"', from=1-2, to=2-2]
		\arrow[from=1-1, to=2-1]
		\arrow["{j^\prime_*}", from=1-1, to=1-2]
		\arrow["{h_*}"', from=2-2, to=3-2]
		\arrow["\lrcorner"{anchor=center, pos=0.125}, draw=none, from=2-1, to=3-2]
		\arrow["\lrcorner"{anchor=center, pos=0.125}, draw=none, from=1-1, to=2-2]
	\end{tikzcd}\]
	where $ h_* $ is as in \ref{lem:AlphaIsEquForLocWeaklyContr} and $ i_* $ is a closed immersion.
	By Lemma \ref{lem:AlphaIsEquForLocWeaklyContr}, we have an equivalence $j^{\prime\prime}_\ast (j^{\prime\prime})^\ast\simeq (-)^{\sh}_{\iota^{\prime\prime}}$, where $\iota^{\prime\prime}\colon h^\ast(\Univ[S])\to \Univ[\YY]$ is the transpose of $\const_{ h_\ast(\Univ[\YY])}\iota$.
	Furthermore, since $ i_* $ is a closed immersion and therefore proper (by Example \ref{ex:closedimmersions}), the upper square is horizontally left adjointable. Thus, we have an equivalence $ j^\prime_* (j^\prime)^* \simeq i^*j^{\prime\prime}_*(j^{\prime\prime})^*i_* $ and
	hence $j^\prime_* (j^\prime)^* \simeq i^*(-)^{\sh}_{\iota^{\prime\prime}} i_* $. Now as $i_\ast$ is proper and therefore compact by Lemma \ref{lem:properImpliesCompact}, we may apply Lemma \ref{lem:compactMorphismCommutesWithSheafification} to deduce $(-)^{\sh}_{\iota^{\prime\prime}}i_\ast\simeq i_{\ast}(-)^{\sh}_{\iota^\prime}$, which yields the claim.
\end{proof}

We are finally ready to prove Proposition \ref{prop:SpecialCaseOfBeckChevalley}:
\begin{proof}[{Proof of Proposition \ref{prop:SpecialCaseOfBeckChevalley}}]
	It suffices to construct a natural equivalence $ p_* j^{\prime}_* (j^{\prime})^*  \simeq j_* j^* p_*  $.
	We pick a local class $S$ in $\BB$, as in Proposition \ref{prop:comparisonofcolimits}, and we let $\iota\colon \Univ[S]\into\Univ[\BB]$ be the associated full subcategory. Furthermore, we let $\iota^\prime\colon p^\ast(\Univ[S])\to \Univ[\XX]$ be the transpose of $\const_{ p_\ast(\Univ[\XX])}\iota\colon \Univ[S]\to p_\ast(\Univ[\XX])$. We then have equivalences $ j_* j^* \simeq (-)^{\sh}_{\iota}$ (by \cite[Proposition~3.2.10.14]{PresTop}) and $ j^{\prime}_* (j^{\prime})^* \simeq (-)^{\sh}_{\iota^\prime}$ (by Proposition \ref{prop:comparisonofcolimits}). Hence the claim follows from Lemma \ref{lem:compactMorphismCommutesWithSheafification}.
\end{proof}

\subsection{The proof of Theorem~\ref{thm:CompactProperBaseChange}}
\label{sec:proof}
We now turn to the proof of the main theorem. We begin with the following small but useful observation:
\begin{lemma}
	\label{lem:leftAdjointabilityDeterminedByGlobalSections}
	Let
	\begin{equation*}
	\begin{tikzcd}
	\I{Q}\arrow[r, "g_\ast"]\arrow[d, "q_\ast"] & \I{P}\arrow[d, "p_\ast"]\\
	\I{Y}\arrow[r, "f_\ast"] & \I{X}
	\end{tikzcd}
	\end{equation*}
	be a commutative square in $\RTop(\BB)$. Then the mate transformation $\phi\colon f^\ast p_\ast \to g_\ast q^\ast$ is an equivalence if and only if it induces an equivalence on global sections.
\end{lemma}
\begin{proof}
	Since the condition is clearly necessary, it suffices to show that it is sufficient too. To that end, we need to show that for any object $A\in\BB$, the horizontal mate $\phi(A)$ of the back square in the commutative diagram
	\begin{equation*}
	\begin{tikzcd}[column sep={5em,between origins}, row sep={3em,between origins}]
	& \I{Q}(A)\arrow[rr, "g_\ast(A)"]\arrow[dd, "q_\ast(A)"', near end]\arrow[dl, "(\pi_A)_\ast"'] && \I{P}(A)\arrow[dl, "(\pi_A)_\ast"]\arrow[dd, "p_\ast(A)"]\\
	\I{Q}(1)\arrow[rr, crossing over, "g_\ast(1)", near end]\arrow[dd, "q_\ast(1)"'] && \I{P}(1)&\\
	&\I{Y}(A)\arrow[rr, "f_\ast(A)"', near start]\arrow[dl, "(\pi_A)_\ast"'] && \I{X}(A)\arrow[dl, "(\pi_A)_\ast"] \\
	\I{Y}(1)\arrow[rr, "f_\ast(1)"'] && \I{X}(1)\arrow[from=uu, "p_\ast(1)", crossing over, near start]
	\end{tikzcd}
	\end{equation*}
	is an equivalence, given that the mate $\phi(1)$ of the front square is one. But since the horizontal mate of both the left and the right square is an equivalence, it follows that $\phi(A)$ is an equivalence when evaluated at any object in the image of $\pi_A^\ast$. Since $\I{X}(A)$ is \'etale over $\I{X}(1)$, every object in $\I{X}(A)$ is a pullback of objects that are contained in the image of $\pi_A^\ast$. Therefore, the claim follows from the fact that $\phi(A)$ is a morphism of left exact functor.
\end{proof}

In order to prove Theorem~\ref{thm:CompactProperBaseChange}, we in particular need to show that compact morphisms are stable under pullback.
In fact it will suffice to prove this in a special case (see Corollary~\ref{cor:PresheafCategoriesCompact}), which we will turn to now.

\begin{lemma}
	\label{lem:pointwiseCompactness}
	Let $f_\ast\colon\XX\to\BB$ be a geometric morphism of $\infty$-topoi.
	Suppose we are given a commutative square
\[\begin{tikzcd}
	{\WW} & \XX \\
	\ZZ & \BB
	\arrow["{f_*}", from=1-2, to=2-2]
	\arrow["{p_*}"', from=2-1, to=2-2]
	\arrow["{g_*}"', from=1-1, to=2-1]
	\arrow["{q_*}", from=1-1, to=1-2]
\end{tikzcd}\]
	whose horizontal mate is an equivalence and such that $g_*$ is compact. Then, for every filtered $\BB$-category $\I{I}$, the functor $p^\ast\colon \BB \to \ZZ$ carries the horizontal mate of the commutative square
	\begin{equation*}
		\begin{tikzcd}
			\XX \arrow[r, "\diag"]\arrow[d, "f_\ast"] & \Fun_{\BB}(\I{I},f_* \Univ[\XX])\arrow[d, "(\Gamma_{f_\ast\Univ[\XX]})_\ast"]\\
			\BB \arrow[r, "\diag"] & \Fun_\BB(\I{I},\Univ[\BB])
		\end{tikzcd}
	\end{equation*}
	to an equivalence.
\end{lemma}
\begin{proof}
	Note that we have a commutative diagram of $ \infty $-topoi
\[\begin{tikzcd}
	& {\Fun_\BB(\I{I}, f_*\Univ[\XX])} && {\Fun_\BB(\I{I}, p_* g_*\Univ[\WW])} \\
	\XX && \WW \\
	& {\Fun_\BB(\I{I}, \Univ[\BB])} && {\Fun_\BB(\I{I}, p_* \Univ[\ZZ])} \\
	\BB && \ZZ
	\arrow["{f_*}"', from=2-1, to=4-1]
	\arrow["{g_*}"{pos=0.3}, from=2-3, to=4-3]
	\arrow["{p_*}", from=4-3, to=4-1]
	\arrow["{(p_*)_*}"{pos=0.3}, from=3-4, to=3-2]
	\arrow["{(q_*)_*}"', from=1-4, to=1-2]
	\arrow["{(f_*)_*}"'{pos=0.7}, from=1-2, to=3-2]
	\arrow["{q_*}"'{pos=0.4}, from=2-3, to=2-1, crossing over]
	\arrow["{(g_*)_*}", from=1-4, to=3-4, crossing over]
	\arrow["\diag"{description}, from=2-1, to=1-2]
	\arrow["\diag"{description}, from=4-1, to=3-2]
	\arrow["\diag"{description}, from=4-3, to=3-4]
	\arrow["\diag"{description}, from=2-3, to=1-4]
\end{tikzcd}\]
	where the horizontal mates of the front and the back square are invertible (the latter using Lemma~\ref{lem:leftAdjointabilityDeterminedByGlobalSections}).
	Furthermore, the adjunction $ p^* \dashv p_* $ allows us to identify the right square with
	\[\begin{tikzcd}
		\WW & {\Fun_{\ZZ}(p^* \I{I}, g_* \Univ[\WW])} \\
		\ZZ & {\Fun_{\ZZ}(p^* \I{I}, \Univ[\ZZ])}
		\arrow["(g_\ast)_\ast", from=1-2, to=2-2]
		\arrow["\diag", from=2-1, to=2-2]
		\arrow["g_\ast", from=1-1, to=2-1]
		\arrow["\diag", from=1-1, to=1-2]
	\end{tikzcd}\]
	whose horizontal mate is invertible since $ g_* $ was assumed to be compact and $ p^* \I{I} $ is filtered.
	Therefore, the functoriality of mates implies that the functor $p^\ast\colon \BB \to \ZZ $ carries the horizontal mate of the left square to the mate of the right square, so an equivalence.
\end{proof}

As a consequence of Lemma~\ref{lem:pointwiseCompactness}, we obtain that compactness can be checked locally on the base in the following strong sense:

\begin{proposition}
	\label{prop:ProperIsH-Local}
	Let $f_\ast\colon\XX\to\BB$ be a geometric morphism of $\infty$-topoi.
	Assume that there exists a family of commutative squares
	\[\begin{tikzcd}
		{\WW_i} & \XX \\
		\ZZ_i & \BB
		\arrow[from=1-2, to=2-2,"f_*"]
		\arrow[from=2-1, to=2-2,"p^i_*"]
		\arrow[from=1-1, to=2-1,"g^i_*"]
		\arrow[from=1-1, to=1-2]
	\end{tikzcd}\]
	whose mate is an equivalence such that the $ (p^i)^* $ are jointly conservative and each $ g^i_* $ is compact.
	Then $ f_* $ is compact.
\end{proposition}
\begin{proof}
	First, let us verify that for any filtered $ \BB $-category $ \I{I} $ the mate of the commutative square
	\[
	\begin{tikzcd}
		f_\ast\Univ[\XX]\arrow[r, "\diag"]\arrow[d, "f_\ast"] & \iFun[\BB](\I{I}, f_\ast\Univ[\XX])\arrow[d, "(\Gamma_{f_\ast\Univ[\XX]})_\ast"]\\
		\Univ[\BB]\arrow[r, "\diag"] & \iFun[\BB](\I{I}, \Univ[\BB])
	\end{tikzcd}
	\]
	is an equivalence.
	By Lemma~\ref{lem:leftAdjointabilityDeterminedByGlobalSections} it suffices to see this on global sections.
	Our assumptions guarantee that we can check that the mate is an equivalence after applying $ (p^i)^* \colon \BB \to \ZZ_i$ for every $ i $.
	But then the claim follows from Lemma~\ref{lem:pointwiseCompactness}.
	Now if $ A \in \BB $ and $ \I{I} $ is a filtered $ \BB_{/A} $-category, we observe that the family of squares obtained by pulling back along $ (\pi_A)_\ast\colon\BB_{/A} \to \BB $ again satisfy the assumptions of the proposition.
	Thus we can replace $ \BB $ by $ \BB_{/A} $ in the first part of the proof and the result follows.
\end{proof}

\begin{corollary}
	\label{cor:PresheafCategoriesCompact}
	Let $ p_\ast \colon \XX \to \BB $ be a compact geometric morphism and let $\I{C} $ be a $ \BB $-category.
	Then the geometric morphism $ (\Gamma_{p_\ast\Univ[\XX]})_\ast\colon\Fun_{\BB}(\I{C}, p_* \Univ[\XX]) \to \Fun_{\BB}(\I{C},\Univ[\BB]) $ is again compact.
\end{corollary}
\begin{proof}
	The core inclusion $\iota\colon\I{C}^\core\into\I{C}$ gives rise to a geometric morphism $\iota_\ast\colon \Over{\BB}{\I{C}^\core}\simeq\Fun_{\BB}(\I{C}^{\core},\Univ[\BB])\to\Fun_{\BB}(\I{C},\Univ[\BB])$ whose left adjoint
	is given by restriction along  $\iota$ and is therefore conservative (which is easily seen using Proposition~\ref{prop:straightening} together with~\cite[Proposition~4.1.18]{Yoneda}).
	Since in the commutative diagram
	\[\begin{tikzcd}
		{\XX_{/p^*(\I{C}^\simeq)}} & { \Fun_{\BB}(\I{C},p_* \Univ[\XX])} & \XX \\
		{\BB_{/\I{C}^\simeq}} & { \Fun_{\BB}(\I{C},\Univ[\BB])} & \BB
		\arrow[from=2-2, to=2-3, "\lim"]
		\arrow["{p_*}"', from=1-3, to=2-3]
		\arrow[from=1-2, to=1-3, "\lim"]
		\arrow[from=1-1, to=2-1]
		\arrow[from=2-1, to=2-2]
		\arrow[from=1-1, to=1-2]
		\arrow[from=1-2, to=2-2]
	\end{tikzcd}\]
	both squares are pullbacks (the one on the right by~\cite[Example~3.2.7.5]{PresTop}), it follows that the left vertical morphism is compact as an \'etale base change of a compact morphism.
	As a consequence, the left square satifies the assumptions of Proposition~\ref{prop:ProperIsH-Local}, which immediately yields the claim.
\end{proof}

\begin{proof}[Proof of Theorem~\ref{thm:CompactProperBaseChange}]
	Suppose that $ p_* \colon \XX \to \BB $ is a compact geometric morphism.
	First, we show that for any pullback square
	\[\begin{tikzcd}
		\ZZ' & \XX \\
		\ZZ & \BB
		\arrow["{p_*}", from=1-2, to=2-2]
		\arrow["{f_*}", from=2-1, to=2-2]
		\arrow["{g_*}", from=1-1, to=1-2]
		\arrow["{q_*}"', from=1-1, to=2-1]
	\end{tikzcd}\]
	in $ \RTopS $ the mate natural transformation $ q_* g^* \to  f^* p_* $ is invertible.
	To see this, we factor the above square as
	\[\begin{tikzcd}
		\ZZ' & {\Fun_\BB(\I{C}^\op,p_\ast\Univ[\XX])} & \XX \\
		\ZZ & {\Fun_\BB(\I{C}^\op,\Univ_\BB)} & \BB.
		\arrow["{p_*}"', from=1-3, to=2-3]
		\arrow["{(p_*)_\ast}", from=1-2, to=2-2]
		\arrow["{j_*}", from=2-1, to=2-2]
		\arrow["{\lim_{\I{C}^\op}}", from=1-2, to=1-3]
		\arrow["{j'_*}", from=1-1, to=1-2]
		\arrow["{q_*}", from=1-1, to=2-1]
		\arrow["{\lim_{\I{C}^\op}}", from=2-2, to=2-3]
	\end{tikzcd}\]
	(using again~\cite[Example~3.2.7.5]{PresTop}).
	It is clear that the mate of the right square is an equivalence, hence it suffices to show the claim for the left square.
	In other words, by Corollary~\ref{cor:PresheafCategoriesCompact} we may reduce to the case where $f_\ast$ is already fully faithful, which follows from Proposition~\ref{prop:SpecialCaseOfBeckChevalley}. 
	
	To complete the proof, we now have to show that given a second pullback
	\[\begin{tikzcd}
		{\WW'} & \ZZ' & \XX \\
		\WW & \ZZ & \BB
		\arrow["{p_*}", from=1-3, to=2-3]
		\arrow["{f_*}", from=2-2, to=2-3]
		\arrow["{g_*}", from=1-2, to=1-3]
		\arrow["{q_*}"', from=1-2, to=2-2]
		\arrow["{r_*}", from=2-1, to=2-2]
		\arrow["{\bar{q}_*}", from=1-1, to=2-1]
		\arrow["{s_*}", from=1-1, to=1-2]
	\end{tikzcd}\]
	in $ \RTopS $ the mate of the left square is an equivalence.
	For this we again use the factorisation from above and consider the diagram
	\[\begin{tikzcd}
		{\WW'} & \ZZ' & {\Fun_\BB(\I{C}^\op,p_\ast\Univ[\XX])} \\
		\WW & \ZZ & {\Fun_\BB(\I{C}^\op,\Univ_\BB)}
		\arrow["{(p_*)_\ast}", from=1-3, to=2-3]
		\arrow["{j_*}", from=2-2, to=2-3]
		\arrow["{j'_*}", from=1-2, to=1-3]
		\arrow["{q_*}"', from=1-2, to=2-2]
		\arrow["{r_*}", from=2-1, to=2-2]
		\arrow["{q^\prime_*}", from=1-1, to=2-1]
		\arrow["{s_*}", from=1-1, to=1-2]
	\end{tikzcd}\]
	By Corollary~\ref{cor:PresheafCategoriesCompact} the geometric morphism $ (p_*)_\ast $ is compact. Together with what we have already shown so far, this implies that both the outer square and the right square is left adjointable.
	As $ j'_* $ is fully faithful it now immediately follows that the left square is also left adjointable, as desired.
\end{proof}

\subsection{Application: proper maps in topology}
\label{sec:Topology}
Recall that a map $p\colon Y\to X$ of topological spaces is called \emph{proper} if it is universally closed, and $p$ is called \emph{separated} if the diagonal $Y\to Y\times_X Y$ is a closed embedding. These are the relative versions of compactness and of being Hausdorff, respectively. Our main goal in this section is to prove the following result about proper separated maps:
\begin{theorem}
	\label{thm:properMapsInTopology}
	Let $p\colon Y\to X$ be a proper and separated map of topological spaces. Then the induced geometric morphism $p_\ast\colon\Shv(Y)\to\Shv(X)$ is proper.
\end{theorem}

\begin{remark}
    A continuous map $p \colon Y \to X$ is separated as soon as $Y$ is Hausdorff.
    Since any completely regular topological space is Hausdorff, it follows that Theorem~\ref{thm:properMapsInTopology} includes \cite[Theorem 7.3.1.6]{htt}.
\end{remark}

\begin{example}
    \label{ex:SeparatedNecessary}
    It follows from \cite[Example C.3.4.1]{johnstone2002} that the separatedness assumption in Theorem~\ref{thm:properMapsInTopology} cannot be dropped.
    We briefly recall the example for the convenience of the reader.
    Consider the topological space $Y$ that is given by taking two copies of the interval $[0,1]$ and identifying both copies of $x$ for $0<x<1$.
    Then $Y$ is compact, but $\Shv(Y)$ is not.
    Indeed, consider the sequence that takes $n\in\mathbb N$ to the sheaf represented by the map $Y_n\to Y$ in which $Y_n$ is given by two copies of $[0,1]$ where we identify both copies of $x$ for $ 2^{-n} < x < 1 - 2^{-n}$.
    We note that all the maps $Y_n \to Y_{n+1}$ and $Y_n \to Y$ are local homeomorphisms, which implies that the colimit of the sheaves represented by $(Y_n\to Y)_{n\in\mathbb N}$ is the sheaf represented by $ \colim_n Y_n = Y$.
    In particular, we have $\Gamma_Y(\colim_n Y_n) = 1$, but since $\colim_n \Gamma_Y (Y_n) = \varnothing$,  the global sections functor $\Gamma_Y$ does not commute with filtered colimits.
\end{example}

Before we prove Theorem~\ref{thm:properMapsInTopology}, let us record that it implies the proper base change theorem in topology, at least for sober spaces:
\begin{corollary}
	\label{cor:properBaseChangeTopology}
	For every pullback square
	\begin{equation*}
		\begin{tikzcd}
			Q\arrow[d, "q"]\arrow[r, "g"] & P\arrow[d, "p"]\\
			Y\arrow[r, "f"] & X
		\end{tikzcd}
	\end{equation*}
	of sober topological spaces in which $p$ is proper and separated, the induced commutative square
	\begin{equation*}
	\begin{tikzcd}
	\Shv(Q)\arrow[d, "q_\ast"]\arrow[r, "g_\ast"] & \Shv(P)\arrow[d, "p_\ast"]\\
	\Shv(Y)\arrow[r, "f_\ast"] & \Shv(X)
	\end{tikzcd}
	\end{equation*}
	is horizontally left adjointable.
\end{corollary}
\begin{proof}
	Using Theorem~\ref{thm:properMapsInTopology}, it suffices to show that the second square is a pullback in $\RTopS$, or equivalently that the underlying square of locales is a pullback. The latter fact follows from combining~\cite[Corollary~3.6]{Johnstone1981} with~\cite[Lemmas~2.1]{Johnstone1981}.
\end{proof}

We now move on to the proof of Theorem~\ref{thm:properMapsInTopology}.
First, let us give an informal outline of our strategy.
We begin with the special case where $X$ is the point, so that $Y$ is a compact Hausdorff space.
Let $\OO(Y)$ be the associated locale of open subsets in $Y$.
It is then a classical fact that $\OO(Y)$ is a \emph{stably compact} locale, i.e. a retract (in the category of locales) of a \emph{coherent locale} (see for example~\cite[\S~C4.1]{johnstone2002}).
In fact, using that $Y$ is compact Hausdorff, one can show that the subposet $\OO_c(Y)\into \OO(Y)$ of opens in $Y$ with compact closure is a distributive lattice, and that moreover every open subset of $Y$ is a union of opens with compact closure.
These elementary observations imply that the poset of ideals $\Id(\OO_c(Y))$ is a coherent locale and that the canonical map $\Id(\OO_c(Y))\to\OO(Y)$ (which takes an ideal to its union) admits a section sending $U\subset Y$ to the ideal $\{V\in \OO_C(Y)~\vert~\overline{V}\subset U\}$.
Moreover, both the map $\Id(\OO_c(Y))\to\OO(Y)$ and its section define morphisms in the category of locales.
Hence, one obtains that $\OO(Y)$ is a retract of a coherent locale.
By Example~\ref{ex:stably_compact_locales} this implies that $\Shv(Y)$ is a compact $\infty$-topos and therefore proper by Theorem~\ref{thm:CompactProperBaseChange}.

The proof of the general case proceeds in exactly the same way.
The only difference is that all steps now have to be carried out internally in $\Shv(X)$.
More specifically, if now $p\colon Y \to X$ is a proper and separated map of topological spaces, we obtain a $\Shv(X)$-category $\I{O}_{X}(Y)$ by means of the sheaf $U\mapsto \OO(p^\ast(U))$ on $X$.
Since $\I{O}_{X}(Y)$ takes values in the $1$-category $\Pos\into \CatS$ of posets and since we have an equivalence $\Fun^{\lim}(\Shv(X)^\op,\Pos)\simeq\Shv[\Pos](X)$, we can equivalently regard $\I{O}_{X}(Y)$ as an internal poset in the $1$-topos $\Shv[\Set](X)$.
As such, this is an example of an internal locale in the sense of~\cite[\S~C1.6]{johnstone2002}.
Now by a result of Johnstone~\cite{Johnstone1981}, the above proof that the locale of opens on a compact Hausdorff is stably compact can be interpreted internally in any $1$-topos.
Consequently, one obtains that $\I{O}_{X}(Y)$ is a stably compact $\Shv[\Set](X)$-locale, i.e. a retract (in the category of internal locales in $\Shv[\Set](X)$) of a \emph{coherent} internal locale.
Furthemore we have shown in \cite[\S~3.4.4]{PresTop} that to any internal locale $\I{L}$ in $\Shv[\Set](X)$ one can functorially associate a $\Shv(X)$-topos $\IShv[\Shv(X)](\I{L})$ and by Example~\ref{ex:stably_compact_locales} this assignment takes stably compact internal locales to compact $\Shv(X)$-topoi.
Finally $\IShv[\Shv(X)](\I{O}_{X}(Y))$ recovers the geometric morphism $p_* \colon \Shv(X) \to \Shv(Y)$ on global sections, so the result follows using Theorem~\ref{thm:CompactProperBaseChange}.

In \cite[\S~3.4.8]{PresTop} we gave a self-contained proof of Johnstones result that is phrased entirely in the language of $\infty$-topoi. Assuming the facts proven there, the proof of Theorem~\ref{thm:properMapsInTopology} is now remarkably short:

\begin{proof}[{Proof of Theorem~\ref{thm:properMapsInTopology}}]
	By \cite[Proposition~3.4.8.6]{PresTop}, $\I{O}_{X}(Y)$ is a \emph{stably compact $\Shv(X)$-locale} in the sense of \cite[Definition~3.4.7.1]{PresTop}.
	Thus it follows from Example~\ref{ex:stably_compact_locales} that $\IShv[\Shv(X)](\I{O}_{X}(Y))$ is a compact $\Shv(X)$-topos.
	As this $\Shv(X)$-topos recovers the geometric morphism $p_* \colon \Shv(X) \to \Shv(Y)$ when passing to global sections (\cite[Corollary~3.4.6.2]{PresTop}), the claim now follows from Theorem~\ref{thm:CompactProperBaseChange}.
\end{proof}

\begin{remark}
	\label{rem:Kosstrategy}
	If one assumes that the topological space $ Y $ is \emph{completely regular} (see \cite[Definition 7.3.1.12]{htt}), one can alternatively apply a number of geometric reduction steps, as in the proof of \cite[Theorem 7.3.16]{htt}, to reduce to the case where $ X = \ast $ and then use that any compact Hausdorff space is a retract of a coherent topological space.
	This proof strategy for Theorem~\ref{thm:properMapsInTopology} was explained to us by Ko Aoki.
	In comparison, 
	Lurie shows that $ \Shv(Y) $ is compact in \cite[Corollary 7.3.4.12]{htt} by using the theory of $ \mathcal{K} $-sheaves.
\end{remark}

\begin{remark}
	If $p\colon Y\to X$ is only assumed to be \emph{locally} proper (see \cite[Definition~2.3]{Schnuerer2016} for a precise definition), the same argumentation as in the proof of Theorem~\ref{thm:properMapsInTopology} shows that the $\Shv(X)$-topos $\IShv[\Shv(X)](\I{O}_{X}(Y))$ is \emph{compactly assembled}. Therefore, by suitably internalising the arguments in~\cite[\S~21.1.6]{lurie2018} (or alternatively those in~\cite{Anel2018a}), one can deduce that $p_\ast$ is \emph{exponentiable} (i.e.\ that $-\times_{\Shv(X)}\Shv(Y)\colon\RTopS\to\RTopS$ has a right adjoint) and that, as a consequence, the stable $\infty$-category $\Shv[\Sp](Y)$ of sheaves of spectra on $Y$ is a dualisable $\Shv[\Sp](X)$-module.
\end{remark}

\subsection{Proper base change with coefficients}
\label{sec:coefficients}
The goal of this section is to discuss a generalisation of Theorem~\ref{thm:CompactProperBaseChange} where we allow coefficients in an arbitrary compactly generated $\infty$-category $\EE$.
The proof is is essentially the same as the one of Theorem~\ref{thm:CompactProperBaseChange}, however this level of generality allows us to apply the result to a wider range of examples.
\begin{definition}
	\label{def:E-Proper}
	Let $ \EE $ be a presentable $ \infty $-category.
	Let $f_\ast\colon \XX\to \BB$ be a geometric morphism of $\infty$-topoi.
	We say that $f_\ast$ is \emph{$ \EE $-proper} if for every diagram
	\[\begin{tikzcd}
		{\WW'} & \WW & \XX \\
		{\ZZ'} & \ZZ & \BB
		\arrow["{f_*}", from=1-3, to=2-3]
		\arrow["{s_*}"', from=2-2, to=2-3]
		\arrow["{g_*}"', from=1-2, to=2-2]
		\arrow["{t_*}", from=1-2, to=1-3]
		\arrow["{s'_*}"', from=2-1, to=2-2]
		\arrow["{t'_*}", from=1-1, to=1-2]
		\arrow["{h_*}"', from=1-1, to=2-1]
	\end{tikzcd}\]
	in $\RTopS$ in which both squares are pullbacks, the square
	\begin{equation*}
		\begin{tikzcd}
			\WW' \otimes \EE \arrow[r, "g_\ast \otimes \EE"] \arrow[d, "p_\ast^\prime \otimes \EE"] & \WW \otimes \EE \arrow[d, "f_\ast \otimes \EE"]\\
			\ZZ' \otimes \EE \arrow[r, "q_\ast \otimes \EE"] & \ZZ \otimes \EE
		\end{tikzcd}
	\end{equation*}
	is horizontally left adjointable.
	Here $ - \otimes - \colon \RPrS \times \RPrS \to \RPrS $ denotes Lurie's tensor product of presentable $ \infty $-categories.
\end{definition}

There is an natural way to enhance Lurie's tensor products to $ \BB $-categories, that we will need to formulate our version of compactness with coefficients:

\begin{construction}
	In \cite[Construction A.0.1]{Colimits} we constructed a functor $ - \otimes \Univ[\BB] \colon \LPr \to \LPr(\BB) $ that sends a presentable $ \infty $-category $ \EE $ to the $ \BB $-category
	\[
	\EE \otimes \Univ[\BB] \colon \BB^\op \to \Cat_{\infty}; \; \; \; A \mapsto \EE \otimes \BB_{/A}.
	\]
	For a presentable $ \BB $-category $ \I{C} $, we can therefore consider the $ \BB $-category $ \I{C} \otimes \EE \coloneqq \I{C} \otimes^\BB (\EE \otimes \Univ[\BB]) $.
	Here $ -\otimes^\BB- $ denotes the tensor product of presentable $ \BB $-categories introduced in \cite[\S~2.6]{PresTop}.
	In particular $ -\otimes \EE $ defines a functor $ \LPr(\BB) \to \LPr(\BB)$.
\end{construction}

\begin{remark}
	\label{rem:linearisingcommuteswithfunctors}
	If $ \I{I} $ is a $ \BB $-category, $ \I{C} $ a presentable $ \BB $-category and $ \EE $ is a presentable $ \infty $-category, it follows from the explicit descripition of the tensor product of presentable $ \BB $-categories \cite[Proposition 2.6.2.11]{PresTop} that we have a canonical equivalence $\iFun(\I{I} , \I{C}) \otimes \EE \simeq \iFun(\I{I},\I{C} \otimes \EE)$.
	In particular it follows from \cite[Proposition~2.6.3.13]{PresTop} that we have an equivalence $\I{C}(A)\otimes \EE\simeq (\I{C}\otimes\EE)(A)$ for every $A\in\BB$.
\end{remark}

\begin{definition}
	Let $ p_* \colon \XX \to \BB $ be a geometric morphism and $ \EE $ a presentable $ \infty $-category.
	Then $ p_* $ is called $ \EE $-proper if $ \Gamma_{p_* \Univ[\XX]}\otimes \EE \colon p_* \Univ_{\XX} \otimes \EE \to \Univ_\BB \otimes \EE $ commutes with filtered colimits.
\end{definition}

We now come to the main result of this section, the $\EE$-linear version of Theorem~\ref{thm:CompactProperBaseChange}:
\begin{theorem}
	\label{thm:ProperBaseChangeWithCoefficients}
	Let $ p_* \colon \mathcal{X} \to \BB $ be a geometric morphism and $ \EE $ a compactly generated $ \infty $-category.
	Then $ p_* $ is $ \EE $-proper if and only if it is $ \EE $-compact.
\end{theorem}

\begin{remark}
	More generally one could define that $ p_* \colon \XX \to \BB $ is $ \I{E} $-compact for a presentable $ \BB $-category $\I{E}$,
	whenever $ p_* \Univ[\XX] \otimes^\BB \I{E} \to \I{E} $ commutes with filtered colimits.
	Similiary one can also define a notion of $ \I{E} $-properness.
	Then the analoguge of Theorem~\ref{thm:ProperBaseChangeWithCoefficients} still holds whenever $ \I{E} $ is compactly generated (in a suitable $\BB$-categorical sense).
	We decided to only prove the result in the case where $ \I{E} = \Univ[\BB] \otimes \EE$, since the proof is slightly less technical and since this case already contains most examples of interest.
\end{remark}

\begin{remark}
	\label{rem:TensorIsLexFunctors}
	Let $ \EE $ be a compactly generated $ \infty $-category and $ \I{X} $ a $ \BB $-topos.
	Then for any $ A \in \BB $ we may identify the tensor product $ \I{X}(A) \otimes \EE $ with the $ \infty $-category $ \Fun^\lex((\EE^{\compact})^{\op},\I{X}(A)) $ (where $\EE^{\compact}\into\EE$ is the full subcategory of compact objects).
	Furthermore, since for any map $s \colon  B \to A $ the transition functors $ s^* \colon \I{X}(A) \to \I{X}(B)$ is a left exact left adjoint, it follows that we may identify the transition map $ s^* \otimes \EE \colon (\I{X} \otimes \EE)(A) \to (\I{X} \otimes \EE)(B)$ with the functor
	\[
	\Fun^\lex((\EE^{\compact})^{\op},\I{X}(A)) \to \Fun^\lex((\EE^{\compact})^{\op},\I{X}(B))
	\]
	given by postcomposition with $ s^* $.
	Now let $ f_* \colon \I{X} \to \I{Y} $ be a geometric morphism of $ \BB $-topoi.
	Since $ f_* $ and $ f^* $ are both left exact if follows as in \cite[Observation 2.9]{haine2021} that the induced morphism $ f_* \otimes \EE $ is given by pointwise postcomposition with $ f_* $ and its left adjoint is given by postcomposition with $ f^* $.
\end{remark}

We begin by establish the $\EE$-linear analogue of Corollary~\ref{cor:PresheafCategoriesCompact}. This requires a few preparations:
\begin{proposition}
	\label{prop:ProperIsH-LocalLinear}
	Let $f_\ast\colon\XX\to\BB$ be a geometric morphism of $\infty$-topoi and let $ \EE $ be a compactly generated $ \infty $-category. 
	Assume that there exists a family of commutative squares
	\[\begin{tikzcd}
		{\WW_i} & \XX \\
		\ZZ_i & \BB
		\arrow[from=1-2, to=2-2,"f_*"]
		\arrow[from=2-1, to=2-2,"p^i_*"]
		\arrow[from=1-1, to=2-1,"(g^i)_*"]
		\arrow[from=1-1, to=1-2]
	\end{tikzcd}\]
	such that for every $A\in\BB$ the functor $ (-\times_{\BB}\Over{\BB}{A}) \otimes \EE $ carries these squares to left adjointable squares, the $ p_i^* \otimes \EE $ are jointly conservative and each $ (g^i_*) $ is $ \EE $-compact.
	Then $ f_* $ is $ \EE $-compact.
\end{proposition}
\begin{proof}
	The proof is essentially the same as the one of Proposition~\ref{prop:ProperIsH-Local}.
	We first check that for every filtered $ \BB $-category $ \I{I} $ the mate of the commutative square
	\[
	\begin{tikzcd}
		f_\ast\Univ[\XX] \otimes \EE \arrow[r, "\diag"]\arrow[d, "f_* \otimes \EE "] & \iFun[\BB](\I{I}, f_\ast\Univ[\XX] \otimes \EE)\arrow[d, "(\Gamma_{f_\ast\Univ[\XX]}\otimes \EE)_\ast "]\\
		\Univ[\BB] \otimes \EE \arrow[r, "\diag"] & \iFun[\BB](\I{I}, \Univ[\BB] \otimes \EE)
	\end{tikzcd}
	\]
	is an equivalence.
	Let us first show that the mate $ \colim_{\I{I}}  (\Gamma_{f_\ast\Univ[\XX]}\otimes \EE)_\ast \to (f_* \otimes \EE) \colim_{\I{I}} $ is an equivalence after passing to global sections.
	For this, it suffices to see that the mate is an equivalence after composing with the maps $p_i^* \otimes \EE \colon \BB \otimes \EE \to \ZZ_i \otimes \EE $ for all $ i $.
	But now the claims follows from an $\EE$-linear version of Lemma~\ref{lem:pointwiseCompactness}, which is proved in exactly the same way.
	To see that the mate is an equivalence after evaluating at $ A \in \BB $ we may replace $ \BB $ by $ \BB_{/A} $ and the above square by its base change along $\pi_A^\ast$ to reduce to the case treated above.
	Finally, we have to see that for every $ A \in \BB $ the functor of $ \BB_{/A} $-categories $ \pi_A^* (f_\ast\Univ[\XX] \otimes \EE) \colon \pi_A^*(f_\ast\Univ[\XX] \otimes \EE) \to \pi_A^*(\Univ[\BB] \otimes \EE)$ commutes with colimits indexed by filtered $ \BB_{/A} $-categories.
	But this follows again from the above after replacing $ \BB $ by $ \BB_{/A} $.
\end{proof}

\begin{remark}
	Note that in Proposition~\ref{prop:ProperIsH-LocalLinear}, we require that the assumptions also hold locally on $ \BB $, while for the version without coefficents (Proposition~\ref{prop:ProperIsH-Local}) this was automatic due to Lemma~\ref{lem:leftAdjointabilityDeterminedByGlobalSections}.
	To illustrate why Lemma~\ref{lem:leftAdjointabilityDeterminedByGlobalSections} may fail when using coefficients, consider the example where $ \EE = \Sub(\SS)\simeq \Delta^1$ is the $ \infty $-category of $(-1)$-truncated spaces.
	Then, a square
	\[\begin{tikzcd}
		{\WW_i} & \YY \\
		\ZZ_i & \BB
		\arrow[from=1-2, to=2-2,"f_*"]
		\arrow[from=2-1, to=2-2,"p^i_*"]
		\arrow[from=1-1, to=2-1,"(g^i)_*"]
		\arrow[from=1-1, to=1-2]
	\end{tikzcd}\]
	being horizontally left adjointable after tensoring with $ \Sub(\SS) $ simply means that the mate transformation is an equivalence on $ (-1) $-truncated objects in $ \YY $.
	However, after passing to a slice $ \XX_{/X} $, the mate transformations now involves $ (-1) $-truncated objects in $ \XX_{/X} $, i.e.\ subobjects of $X$. These need not be $ (-1) $-truncated in general, therefore there is no reason for the mate transformation to be an equivalence.
\end{remark}

\begin{remark}
	The proof of Proposition~\ref{prop:ProperIsH-LocalLinear} shows that more generally we do not need the existence of such squares for every $A \in \BB $, but it suffices to find these for a set of objects $ A_i \in \BB $ that generates $ \BB $ under colimits.
\end{remark}

\begin{lemma}
	\label{lem:straighteningFunctorialityBaseTopos}
	For every $\BB$-category $\I{C}$ and every geometric morphism $f_\ast\colon \XX\to\BB$, the straightening equivalences $\Fun(\I{C},\Univ[\BB])\simeq\LFib_{\BB}(\I{C})$ and $\Fun_{\BB}(\I{C}, f_\ast\I{C})\simeq\Fun_{\XX}(f^\ast\I{C},\Univ[\XX])\simeq\LFib_{\XX}(f^\ast\I{C})$ fit into a commutative square
	\begin{equation*}
		\begin{tikzcd}
			\Fun_{\BB}(\I{C},\Univ[\BB])\arrow[r, "\simeq"]\arrow[d, "\const_{ f_\ast\Univ[\XX]}"] & \LFib_{\BB}(\I{C})\arrow[d, "f^\ast"]\\
			\Fun_{\BB}(\I{C}, f_\ast\Univ[\XX])\arrow[r, "\simeq"] & \LFib_{\XX}(f^\ast\I{C})
		\end{tikzcd}
	\end{equation*}
	where $f^\ast$ is the restriction of $f^\ast\colon \Cat(\BB)_{/\I{C}}\to\Cat(\XX)_{/f^\ast\I{C}}$ to left fibrations. Moreover, this commutative square is natural in $\I{C}$.
\end{lemma}
\begin{proof}
	This is shown in exactly the same fashion as~\cite[Lemma~4.6.4]{Cocartesian}.
\end{proof}

\begin{corollary}
	\label{cor:PresheafCategoriesCompactLinear}
	Let $ f_\ast\colon \XX \to \BB $ be an $ \EE $-compact geometric morphism and $\I{C} $ a $ \BB $-category.
	Then the geometric morphism $(\Gamma_{f_\ast\Univ[\XX]})_\ast\colon \Fun_{\BB}(\I{C}, f_* \Univ[\XX]) \to \Fun_{\BB}(\I{C},\Univ[\BB]) $ is $ \EE $-compact.
\end{corollary}
\begin{proof}
	Let $ F \in \Fun_{\BB}(\I{C},\Univ[\BB])$ be an arbitrary functor and  let us set $G= (\const_{ f_\ast\Univ[\XX]})_\ast(F)$. Furthermore, let $\Under{\I{C}}{F}\to\I{C}$ be the left fibration associated to $F$ via the straightening equivalence. We then deduce from Lemma~\ref{lem:straighteningFunctorialityBaseTopos} and the fact that left fibrations satisfy the left cancellation property (being determined by a factorisation system) that we have a commutative diagram
	\begin{equation*}
		\begin{tikzcd}
			\Over{\Fun_{\BB}(\I{C}, \Univ[\BB])}{F}\arrow[d, "(\const_{ f_\ast\Univ[\XX]})_\ast"]\arrow[r, "\simeq"] & \LFib_{\BB}(\I{C})_{/(\Under{\I{C}}{F})}\arrow[d, "f^\ast"] \arrow[r, "\simeq"] & \LFib_{\BB}(\Under{\I{C}}{F})\arrow[r, "\simeq"]\arrow[d, "f^\ast"] & \Fun_{\BB}(\Under{\I{C}}{F}, \Univ[\BB])\arrow[d, "(\const_{ f_\ast\Univ[\XX]})_\ast"]\\
			\Over{\Fun_{\BB}(\I{C}, \Univ[\BB])}{G}\arrow[r, "\simeq"] & \LFib_{\XX}(f^\ast\I{C})_{/f^\ast(\Under{\I{C}}{F})}\arrow[r, "\simeq"] & \LFib_{\XX}(f^\ast(\Under{\I{C}}{F}))\arrow[r, "\simeq"] & \Fun_{\BB}(\Under{\I{C}}{F}, f_\ast\Univ[\XX])
		\end{tikzcd}
	\end{equation*}
	which is natural in $\I{C}$.
	Thus, by passing to right adjoints, the base change of $(\Gamma_{f_\ast\Univ[\XX]})_\ast$ along the geometric morphism $(\pi_F)_\ast\colon \Over{\Fun_{\BB}(\I{C},\Univ[\BB])}{F}\to\Fun_{\BB}(\I{C},\Univ[\BB])$ can be identified with with the geometric morphism $(\Gamma_{f_\ast\Univ[\XX]})\colon\Fun_{\BB}(\Under{\I{C}}{F}, f_\ast\Univ[\XX])\to\Fun_{\BB}(\Under{\I{C}}{F}, \Univ[\BB])$. Also, the base change of the right adjoint of the restriction functor $\Fun_{\BB}(\I{C},\Univ[\BB])\to \Fun_{\BB}(\I{C}^\core,\Univ[\BB])\simeq\Over{\BB}{\I{C}^\core}$ along $(\pi_A)_\ast$ can be identified with the right adjoint of the restriction $\Fun_{\BB}(\Under{\I{C}}{F},\Univ[\BB])\to\Over{\BB}{(\Under{\I{C}}{F})^\core}$ (using that the pullback of $\Under{\I{C}}{F}\to\I{C}$ along $\I{C}^\core\to \I{C}$ is $(\Under{\I{C}}{F})^\core$, see~\cite[Corollary~4.1.16]{Yoneda}).
	Consequently, we conclude that the pullback square
	\[\begin{tikzcd}
		{\XX_{/f^*\I{C}^\simeq}} & {\Fun_{\BB}(\I{C},f_* \Univ[\XX])} \\
		{\BB_{/\I{C}^\simeq} } & {\Fun_{\BB}(\I{C},\Univ[\BB])}
		\arrow[from=1-2, to=2-2]
		\arrow[from=2-1, to=2-2]
		\arrow[from=1-1, to=2-1]
		\arrow[from=1-1, to=1-2]
	\end{tikzcd}\]
	satisfies the assumptions of Proposition~\ref{prop:ProperIsH-LocalLinear}.
	Thus the claim follows.
\end{proof}

\begin{proof}[Proof of Theorem~\ref{thm:ProperBaseChangeWithCoefficients}:]
	By Remark~\ref{rem:linearisingcommuteswithfunctors}, the same proof as in Lemma~\ref{lem:properImpliesCompact} shows that an $ \EE $-proper morphism is $ \EE $-compact. Hence it remains to prove the converse.
	By Corollary~\ref{cor:PresheafCategoriesCompactLinear}, the same reduction steps as in the proof of Theorem~\ref{thm:CompactProperBaseChange} imply that it suffices to see that for every pullback square of $ \BB $-topoi
	\[\begin{tikzcd}
		{\ZZ'} & {\XX} \\
		{\ZZ} & {\BB}
		\arrow[from=1-2, to=2-2,"p_*"]
		\arrow[from=2-1, to=2-2,"j_*"]
		\arrow[from=1-1, to=2-1]
		\arrow[from=1-1, to=1-2,"j'_*"]
	\end{tikzcd}\]
	in which $j_\ast$ is fully faithful, the square is left adjointable after tensoring with $\EE$.
	By Remark~\ref{rem:TensorIsLexFunctors}, it suffices to see that the square
	\[\begin{tikzcd}
		{\Fun^{\lex}((\EE^{\compact})^{\op},\XX)} & {\Fun^{\lex}((\EE^{\compact})^{\op},\XX)} \\
		{\Fun^{\lex}((\EE^{\compact})^{\op},\BB)} & {\Fun^{\lex}((\EE^{\compact})^{\op},\BB)}
		\arrow["{(j'_*j'^*)_*}", from=1-1, to=1-2]
		\arrow["{(p_*)_*}", from=1-2, to=2-2]
		\arrow["{(p_*)_*}"', from=1-1, to=2-1]
		\arrow["{(j_* j^*)_*}"', from=2-1, to=2-2]
	\end{tikzcd}\]
	commutes.
	We pick a local class $S$ of maps in $\BB$ as in Proposition~\ref{prop:comparisonofcolimits}, so that we obtain equivalences $j'_* j'^* \simeq (-)^{\sh}_{\iota^\prime} $ and $j_\ast j^\ast \simeq (-)^{\sh}_{\iota}$. Now since the inclusions $\Fun^{\lex}((\EE^{\compact})^\op, \XX)\into\Fun((\EE^{\compact})^\op, \XX)$ and $\Fun^{\lex}((\EE^{\compact})^\op, \BB)\into\Fun((\EE^{\compact})^\op, \BB)$ preserve filtered colimits and since colimits in functor $\infty$-categories are computed object-wise, it follows that $(j'_* j'^*)_\ast$ and $(j_* j^*)_\ast$ are given by the $\kappa$-fold iteration of postcomposition with the functors $(-)^{+}_{\iota^\prime}$ and $(-)^{+}_{\iota}$, respectively. Therefore, it suffices to provide an equivalence $(p_\ast(-)^{+}_{\iota^\prime})_* \simeq ((-)^+_{\iota} p_*)_*$.
	
	To obtain such an equivalence, note that Remark~\ref{rem:linearisingcommuteswithfunctors} implies that we may identify the map $\colim_{\Univ_S^\op}\otimes\EE$ with $\colim_{\Univ_S^\op}\colon \iFun(\Univ[S]^\op, \Univ\otimes \EE)\to \Univ\otimes\EE$. Therefore, we deduce that postcomposition with $(-)^{+}_{\iota}$ is equivalently given by the composition
	\begin{equation*}
		\Fun^{\lex}((\EE^{\compact})^\op, \BB)\to \Fun^{\lex}((\EE^{\compact})^\op, \PSh[\BB](\Univ[S]))\simeq \Fun_{\BB}(\Univ[S]^\op, \Univ\otimes \EE )\xrightarrow{\colim} \Fun^{\lex}((\EE^{\compact})^\op, \BB)
	\end{equation*}
	in which the first functor is given by postcomposition with (the global sections of) $\map{\Univ}(\iota(-),-)$.
	Similarly, postcomposition with $(-)^+_{\iota^\prime}$ can be identified with the composition
	\begin{equation*}
		\Fun^{\lex}((\EE^{\compact})^\op, \XX)\to \Fun^{\lex}((\EE^{\compact})^\op, \Fun_{\BB}(\Univ[S]^\op, p_\ast\Univ[\XX]))\simeq \Fun_{\BB}(\Univ[S]^\op, p_\ast\Univ[\XX]\otimes \EE )\xrightarrow{\colim} \Fun^{\lex}((\EE^{\compact})^\op, \XX)
	\end{equation*}
	where the first functor is given by postcomposition with (the global sections of) $\ihom_{p_\ast\Univ[\XX]}(\const_{ p_\ast\Univ[\XX]}\iota(-),-)$ (since this is precisely the map we obtain when composing $\Gamma_{\XX}(\map{\Univ[\XX]}(\iota^\prime(-),-))\colon \XX\to \PSh[\XX](p^\ast\Univ[S])$ with the equivalence $\PSh[\XX](p^\ast\Univ[S])\simeq\Fun_{\BB}(\Univ[S]^\op, p_\ast\Univ[\XX])$).
	Thus, since $\Gamma_{p_* \Univ[\XX]}\otimes \EE$ commutes with $\colim_{\Univ[S]^\op}$, it is enough to provide a commutative diagram
	\begin{equation*}
		\begin{tikzcd}[column sep=11em]
			\XX\arrow[d, "p_\ast"]\arrow[r, "{\ihom_{p_\ast\Univ[\XX]}(\const_{ p_\ast\Univ[\XX]}\iota(-),-)}"] & \Fun_{\BB}(\Univ[S]^\op, p_\ast\Univ[\XX])\arrow[d, "(\Gamma_{p_* \Univ[\XX]})_\ast"]\\
			\BB\arrow[r, "{\map{\Univ[\XX]}(\iota(-),-)}"] & \PSh[\BB](\Univ[S]),
		\end{tikzcd}
	\end{equation*}
	which is evident from \cite[Remark~3.2.10.13]{PresTop}.
\end{proof}

\begin{example}
	\label{ex:EtaleCohomology}
	For a scheme $ X $ let us denote by $ X_{\et}^{\hyp}$ the $ \infty $-topos of \'etale hypersheaves of spaces on $ X $.
	If $ f \colon X \to S $ is a proper morphism of schemes, then the geometric morphism $ f_* \colon X_{\et}^{\hyp} \to S_{\et}^{\hyp} $ is $ \mathbf{D}(R) $-proper for any torsion ring $ R $.
	In fact, since $ X_{\et}^{\hyp}$ has enough points by \cite[Theorem A.4.0.5]{lurie2018}, the family of all points $ \bar{s}_* \colon \mathcal{S} \to  X_{\et}^{\hyp} $ yields a family of jointly conservative functors $ \bar{s}^* \otimes \mathbf{D}(R) $.
	Furthermore, proper base change for unbounded derived categories of \'etale sheaves (see \cite[Theorem 1.2.1]{cisinski2016}) implies that the squares
	\[\begin{tikzcd}
		{X_{\bar{s},\et}^{\hyp}} & {X_{\et}^{\hyp}} \\
		{\mathcal{S}} & {S_{\et}^{\hyp}}
		\arrow[from=2-1, to=2-2]
		\arrow[from=1-2, to=2-2]
		\arrow[from=1-1, to=2-1]
		\arrow[from=1-1, to=1-2]
	\end{tikzcd}\]
	are left adjointable after applying $ - \otimes \mathbf{D}(R) $.
	Finally, \cite[Corollary 1.1.15]{cisinski2016} implies that $ X_{\bar{s},\et}^{\hyp} $ is $ \mathbf{D}(R) $-compact so that we may apply Proposition~\ref{prop:ProperIsH-LocalLinear} and Theorem~\ref{thm:ProperBaseChangeWithCoefficients} to conclude that $ f_* $ is $ \mathbf{D}(R) $-proper.
\end{example}

\begin{definition}
	We call a geometric morphism $f_* \colon \YY \to \XX$ $n$-proper if it is $\mathcal{S}_{\leq n}$-proper, where $\SS_{\leq n}$ denotes the $\infty$-category of $n$-truncated spaces.
	We call $f_*$ almost proper if it is $n$-proper for all $n$.
\end{definition}

\begin{example}
	Recall that by \cite[Example 4.8.1.22]{Lurie2017} one may identify $ \XX \otimes \mathcal{S}_{\leq n} \simeq \XX_{\leq n}$.
	Thus it follows from \cite[Proposition A.2.3.1]{lurie2018} and Theorem~\ref{thm:ProperBaseChangeWithCoefficients} that for an $n$-coherent $\infty$-topos $\XX$ the geometric morphism $ \Gamma_* \colon \XX \to \mathcal{S}$ is $n$-proper.
	In particular it is almost proper if $\XX$ is coherent.
	However it is not proper in general (see Remark \ref{rem:Spec(R)Counterexample}).
\end{example}

\begin{example}
	A geometric morphism $ f_* \colon \XX \to \BB $ is $ \operatorname{Set} $-proper if and only if the underlying morphism of $ 1 $-topoi is tidy in the sense of \cite[\S3]{Moerdijk2000}.
\end{example}

\bibliographystyle{halpha}
\bibliography{references.bib}

\end{document}